\documentclass[12pt,english,a4paper]{article}
\usepackage{fullpage}
\usepackage{authblk}

\usepackage[english]{babel}    % with explicit language
\usepackage{amsmath,amsthm}           % ams mathematical stuff
\usepackage[utf8]{inputenc}    % smart input of funny chars
\usepackage{xspace}            % better spacing after macros
\usepackage{tikz-cd}            % for commutative diagrams
\usepackage{amssymb}           % ams symbols
\usepackage{enumitem}
\usepackage{mathtools}         % some nice tricks for stackrel
\usepackage{mathrsfs}          % nice calligraphic font
\usepackage{stmaryrd}          % more brackets
\usepackage{eqparbox}
\usepackage{aliascnt}          % TEST alias counter for ntheorem environments
\usepackage[expansion=false    % no font expansion
           ]{microtype}        % only protrusion
\usepackage[%backref=page,      % backrefs in the bibliography
           final=true,         % always treat as final
           pdfpagelabels,      % use pdf page labels
           colorlinks
           ]{hyperref}         % hyperrefs are cool!
\usepackage[shortcuts]{extdash}

\newcommand{\argument}{{-}}%{{}\cdot{}}

\newcommand{\Pol}{\operatorname{\mathcal P}}
\newcommand{\Hol}{\operatorname{\mathcal O}}
\newcommand{\Smooth}{\mathcal C^\infty}
\newcommand{\Hom}{\operatorname{Hom}}
\newcommand{\Ana}{\mathcal{A}}
\newcommand{\RE}{\mathrm{Re}}

\newcommand{\E}{\mathrm e}
\newcommand{\I}{\mathrm i}
\newcommand{\cc}[1]{\overline{#1}}
\let\origphi\phi
\renewcommand{\phi}{\varphi}

\newcommand{\TR}[1][R]{\mathrm T \mkern-2mu_{#1}}
\newcommand{\TRtop}{$\TR$\-/topology}
\newcommand{\realDisc}[1][r]{\mathbb I_{#1}}
\newcommand{\wickDisc}[1][r]{\mathbb J_{#1}}
\newcommand{\complexDisc}[1][r]{\mathbb D_{#1}}
\newcommand{\antidiag}[1][d]{\mathbb J^{#1}}

\newcommand{\llrr}[1]{\llbracket #1 \rrbracket}

\newcommand{\abs}[2][]{#1\lvert#2#1\rvert}
\newcommand{\norm}[2][]{#1\lVert#2#1\rVert}
\newcommand{\dualpairing}[3][]{#1\langle#2,#3#1\rangle}
\newcommand{\triplenorm}[2][]{{#1\vert\kern-0.25ex#1\vert\kern-0.25ex#1\vert #2#1\vert\kern-0.25ex#1\vert\kern-0.25ex#1\vert}}
\newcommand{\set}[2][]{#1\{#2#1\}}

\newcommand{\MG}{\mathrm{MG}}
\newcommand{\weak}[1]{#1_{\MG}}
\newcommand{\weakPol}{\weak{\Pol}}
\newcommand{\weakNorm}[3][]{\norm[#1]{#2}_{#3}^{\MG}}

\newcommand{\punkt}{\,.}
\newcommand{\komma}{\,,}

\newcommand{\hatotimes}{\mathbin{\widehat{\otimes}}}

\newcommand{\varstar}{\ast}

\numberwithin{equation}{section}
\renewcommand{\theequation}{\thesection.\arabic{equation}}

\newtheorem{theorem}[equation]{Theorem}
\newtheorem{lemma}[equation]{Lemma}
\newtheorem{proposition}[equation]{Proposition}
\newtheorem{corollary}[equation]{Corollary}

\theoremstyle{definition}
\newtheorem{definition}[equation]{Definition}
\newtheorem{convention}[equation]{Convention}
\newtheorem{notation}[equation]{Notation}
\newtheorem{example}[equation]{Example}
\newtheorem{remark}[equation]{Remark}

\newcommand{\refitem}[1]{(\ref{#1})}%{\textit{\ref{#1}.)}}

\title{Strict quantization of polynomial Poisson structures}
\author[1,2]{Severin Barmeier\footnote{s.barmeier@gmail.com}}
\author[3]{Philipp Schmitt\footnote{schmitt@math.uni-hannover.de}}
\affil[1]{\small Mathematical Institute, University of Cologne, Weyertal 86-90, 50931 Köln, Germany}
\affil[2]{\small Mathematical Institute, University of Freiburg, Ernst-Zermelo-Str.~1, 79104 Freiburg i.~Br., Germany}
\affil[3]{\small Institute of Analysis, Leibniz University Hannover, Welfengarten 1, 30167 Hannover, Germany}
\date{}

\begin{document}

\maketitle

\begin{abstract}%
We show how combinatorial star products can be used to obtain strict deformation quantizations of 
polynomial Poisson structures on $\mathbb R^d$, generalizing known results for 
constant and linear Poisson structures to polynomial Poisson structures of arbitrary degree.
We give several examples of nonlinear Poisson structures and construct explicit formal star 
products whose deformation parameter can be evaluated to any real value of $\hbar$, giving strict 
quantizations on the space of analytic functions on $\mathbb R^d$ with infinite radius of 
convergence.

We also address further questions such as continuity of the classical limit $\hbar \to 0$, 
compatibility with $^*$\=/involutions, and the existence of positive linear functionals. The latter 
can be used to realize the strict quantizations as $^*$\=/algebras of operators on a pre-Hilbert 
space which we demonstrate in a concrete example.
\end{abstract}

\setcounter{tocdepth}{2}
\tableofcontents

\section{Introduction}

The convergence of formal star products is arguably one of the most important outstanding issues in the deformation quantization programme initiated by Bayen--Flato--Frønsdal--Lichnerowicz--Sternheimer \cite{bayen}. Deformation quantization aims to reverse-engineer a quantum mechanical observable algebra from the classical observable algebra obtained in the limit $\hbar \to 0$. For this process, one starts with a formal star product $f \star g = fg + \sum_{n \geq 1} t^n B_n (f, g)$, the purported perturbative expansion around $\hbar = 0$ of the algebras obtained from a quantum observable algebra by letting $\hbar \to 0$, and views it as a formal deformation of the commutative algebra of classical observables. The formal deformation parameter $t$ stands in for the (reduced) Planck constant $\hbar$ and the existence of formal star products was shown for arbitrary Poisson manifolds by M.~Kontsevich \cite{kontsevich1}. As $\hbar$ is not a formal parameter but a dimensional constant, the physical interpretation of deformation quantization depends on the existence of a ``strict'' quantization \cite{waldmann2}, which within the deformation quantization programme should be obtained by evaluating the formal deformation parameter $t$ to the physical value of $\hbar$. For this evaluation to make sense, the formal star product should be given by convergent series, but studying convergence questions for star products is already a nontrivial undertaking for constant or linear Poisson structures on $\mathbb R^d$ \cite{beiserwaldmann,espositostaporwaldmann,waldmann1}. Notably, the convergence of star products fails for {\it any} nontrivial Poisson structure when considering the space $\Smooth (\mathbb R^d)$ of all smooth complex-valued functions on $\mathbb R^d$.

A general strategy for addressing the problem of convergence in deformation quantization was formulated by S.~Beiser and S.~Waldmann in \cite{beiserwaldmann}. Given a Poisson manifold $X$, one considers a suitable {\it subspace} of the space of smooth functions on $X$ for which the star product is well defined for complex values of $\hbar$. For example, when $X = \mathbb R^d$ carries a polynomial Poisson structure, one may consider the subspace $\Pol (\mathbb R^d) \subset \Smooth (\mathbb R^d)$ of polynomial functions and work with a formal star product which converges on $\Pol (\mathbb R^d)$, giving rise to an associative product $\star_\hbar \colon \Pol(\mathbb R^d) \times \Pol(\mathbb R^d) \to \Pol(\mathbb R^d)$. One then completes $\Pol(\mathbb R^d)$ with respect to a suitable (locally convex) topology for which $\star_\hbar$ is continuous. If done right, the completion will contain many more physically interesting functions, such as exponential functions, say. For constant and linear Poisson structures, one can work with the Moyal--Weyl and Gutt star products, respectively. In this case, convergence on polynomials is immediate, since for all $f, g \in \Pol (\mathbb R^d)$ their formal star product $f \star g$ is a polynomial in the formal parameter which can be evaluated to any value of $\hbar$. The continuity of star products, and the properties of the algebra obtained by completion, was studied successfully for constant and linear Poisson structures \cite{waldmann1,espositostaporwaldmann}, also in infinite-dimensional, field-theoretic \cite{schoetzwaldmann} and ``global'' settings, such as on coadjoint orbits of Lie groups \cite{krausrothschoetzwaldmann,schmitt,schmittschoetz} or on cotangent bundles of Lie groups \cite{heinsrothwaldmann} (see \cite{waldmann2} for a survey). Yet, although constant and linear Poisson structures are important classes of Poisson structures, one cannot expect them to cover all physically relevant phase spaces.

In this article, we develop an approach to deal with the issue of convergence and continuity for 
star products quantizing {\it nonlinear} polynomial Poisson structures on $\mathbb R^d$. For 
nonlinear Poisson structures one cannot use the Moyal--Weyl and Gutt star products, so one needs a 
formal star product that can be shown to converge on polynomials. Although Kontsevich's universal 
quantization formula \cite[\S 2]{kontsevich1} can be made very explicit, it is only a finite sum on 
polynomials for constant or linear Poisson structures \cite{kontsevich1,dito}.
For nonlinear Poisson structures, the asymptotics of the multiple zeta values appearing as weights 
of certain graphs in the formula (see \cite{bankspanzerpym}) make the convergence properties of the 
Kontsevich star product on polynomials difficult to determine and Kontsevich's conjecture on the 
convergence of the Kontsevich star product \cite[Conj.~1]{kontsevich2} is still widely open (cf.\ 
\cite[\S 1.1]{bankspanzerpym}).

We work with the combinatorial star products introduced in \cite{barmeierwang} via natural higher 
structures on the Koszul complex, which can be used to produce explicit formulae for quantizations 
of polynomial Poisson structures (see \S\ref{subsection:combinatorial} and \S\ref{subsec:combinatorialinvolutions}) for which convergence on polynomials can be shown directly (see \S\ref{subsec:convergence}). These formulae can then be used 
in continuity estimates and the star product can be extended from the space of polynomial functions 
to larger function spaces (see \S\ref{sec:continuity}). To this end we work with a range of 
locally convex topologies on $\Pol (\mathbb R^d)$ which are adapted to the various Poisson structures at hand, notably the MacGyver topology (Definition \ref{def:norms:general}) and the $\TR$\=/topology (Definition \ref{definition:TR}).

For quantizations of polynomial Poisson structures satisfying a certain finiteness condition on the associated combinatorial star product, we prove the following general result.

\begin{theorem}[Theorem \ref{theorem:continuity:general}]
\label{theorem:main1}
	Let $\star$ be a combinatorial star product quantizing a polynomial Poisson structure $\eta$ on $\mathbb R^d$ and assume:
	\begin{enumerate}[label={\textup{(\alph*)}}]
		\item For any $1 \leq i,j \leq d$ we have
		\begin{equation}
			x_i \star x_j = \sum_{K \in \mathbb N_0^d} q_{i,j,K} x^K
		\end{equation}
		where $q_{i,j,K} \in \mathbb C \llrr{t}$ are power series expansions of holomorphic functions defined on an open 
		neighbourhood $\Omega$ of $0 
		\in \mathbb C$, only finitely of which are non-zero.
		\item
		There is a constant $\alpha$ (independent of $K$ and $L$) such that at most $\alpha (\abs K + \abs L)^2$ many reductions are needed to compute $x^K \star x^L$.
		\item There is a constant $\beta$ (independent of $K$ and $L$) such that $x^K \star x^L$ is a sum of monomials of order not greater than $\beta (\abs K + \abs L)$.
	\end{enumerate}
	Then evaluating $t \mapsto \hbar$ for any $\hbar \in \Omega$, 
	the resulting product $\star_\hbar$ is continuous with respect to the MacGyver topology on $\Pol (\mathbb R^d)$ and extends uniquely to a continuous product on the completion $\weak{\widehat{\Pol}{}}(\mathbb R^d)$ whose Taylor expansion around $\hbar = 0$ recovers the formal combinatorial star product $\star$.
	
	Moreover, $(\weak{\widehat{\Pol}} (\mathbb R^d), \star_\hbar)$ is a strict deformation quantization of $(\mathbb R^d, \eta)$.
\end{theorem}

Theorem \ref{theorem:main1} provides strict deformation quantizations for a large class of polynomial Poisson structures --- in particular for all examples given in this article. However, the completion $\weak{\widehat{\Pol}} (\mathbb R^d)$ is not as large as one might hope (cf.\ Remark \ref{remark:smallCompletion}). In concrete examples much stronger results can be obtained, as summarized in a slightly simplified form in the following theorem.

\begin{theorem}
\label{theorem:main2}
Let $\eta$ be any of the following nonlinear Poisson structures:
\begin{enumerate}[label={\textup{(\alph*)}}]
\item the log-canonical Poisson structure on $\mathbb R^d$ given by $\{ x_j, x_i \}_\eta = x_i x_j$ for $1 \leq i < j \leq d$
\item the exact Poisson structure on $\mathbb R^3$ associated to the function $-xyz - \frac{1}{N+1} x^{N+1}$ for any fixed $N \in \mathbb N_0$
\item the Poisson structure on $\mathbb R^3$ given by the bivector field $\eta = (yz + x^N) \frac{\partial}{\partial z} \wedge \frac{\partial}{\partial y}$ for any fixed $N \in \mathbb N_0$
\item the Poisson structure on $\mathbb R^2$ given by $\{ y, x \}_\eta = \frac12 (x^2 + y^2)$
\item the Poisson structure on $\mathbb R^2$ given by $\{ y, x \}_\eta = x y + c$ for any constant $c \in \mathbb R$.
\end{enumerate}
Then there exists a formal quantization $\star$ of $\eta$ which has the following convergence and continuity properties:
\begin{enumerate}
\item $\star$ converges on the algebra $\Pol(\mathbb R^d)$ of polynomial functions on $\mathbb R^d$ when evaluating the formal deformation parameter $t$ to any $\hbar \in \mathbb C$.
\item Evaluating $t \mapsto \hbar \in [0, \infty)$, the resulting strict star product $\star_\hbar$ is continuous with respect to the $\TR[0]$\=/topology on $\Pol (\mathbb R^d)$.
\item $\star_\hbar$ extends continuously to the completion of $\Pol (\mathbb R^d)$, which coincides with the space $\Ana (\mathbb R^d)$ of analytic functions with infinite radius of convergence.
\end{enumerate}
Moreover, $(\Ana (\mathbb R^d), \star_\hbar)$ is a strict deformation quantization of $(\mathbb R^d, \eta)$.
\end{theorem}

The algebras $(\Ana (\mathbb R^d), \star_\hbar)$ obtained in Theorem \ref{theorem:main2} contain many well-behaved functions, such as exponential functions, and enjoy many nice properties. For example, for the log-canonical Poisson structure the algebras are locally multiplicatively convex. In this regard they are in fact better behaved than the algebras obtained by similar methods for the Moyal--Weyl or Gutt star products quantizing constant or linear Poisson structures. A particularly surprising case is the last Poisson structure in Theorem \ref{theorem:main2} which can be quantized to the quantum Weyl algebra $\mathbb C \langle x, y \rangle / (y x - \E^{\I \lambda \hbar} x y - \I \hbar)$ for $\lambda > 0$. This algebra can be completed to a much larger strict quantization than the standard Weyl algebra which corresponds to the case $\lambda = 0$ (see Example \ref{example:quantumWeyl}).

We also study further properties relevant to the physical interpretation of the convergence and continuity results, namely the compatibility with \emph{$^*$\=/involutions} (\S\ref{subsec:combinatorialinvolutions}) and the existence of \emph{positive linear functionals} (\S\ref{subsec:continuity:wick}) which allow one to represent strict quantizations as algebras of operators on a (pre)Hilbert space.

\section{Star products}
\label{sec:starproducts}

In this section we briefly recall the relevant background for star products in \S\ref{subsec:basic} and review the notion of combinatorial star products in \S\ref{subsection:combinatorial}. In \S\ref{subsec:combinatorialinvolutions} we study the compatibility of combinatorial star products with $^*$\=/involutions and in \S\ref{subsec:convergence} we prove convergence results for combinatorial star products used for continuity estimates in \S\ref{sec:continuity}.

\subsection{Physical background and basic notions}
\label{subsec:basic}

\paragraph{Observables} Recall that a Poisson manifold $(X, \eta)$ is given by a smooth manifold $X$ and a smooth bivector field $\eta$ on $X$ whose associated Poisson bracket $\{ \argument {,} \argument \}_\eta$ satisfies the Jacobi identity.\footnote{For two smooth functions $f, g \in \Smooth (X)$, the Poisson bracket is given by $\{ f, g \}_\eta = \langle \eta, \mathrm d f \wedge \mathrm d g \rangle$, where $\langle \argument {,} \argument \rangle$ is the pairing between bivector fields and $2$-forms and $\mathrm d$ is the exterior derivative.} The algebra of smooth functions on a Poisson manifold $(X, \eta)$ can be viewed as a Poisson algebra of classical observables, where $X$ is the phase space of a classical mechanical system and the Poisson bracket $\{ \argument {,} \argument \}_\eta$ encodes the time evolution via Hamilton's equations of motion. Deformation quantization after Bayen--Flato--Frønsdal--Lichnerowicz--Sternheimer \cite{bayen} aims to produce an associative algebra of quantum observables from such a Poisson algebra of classical observables by deforming the usual pointwise commutative product of smooth functions to an associative star product $\star$. The time evolution of quantum observables is governed by Heisenberg's equation of motion
\begin{equation*}
%\label{heisenberg}
\frac{\mathrm d}{\mathrm d \tau} A (\tau) = \frac{1}{\I \hbar} [H, A (\tau)]
\end{equation*}
where $A (\tau)$ is a time-dependent quantum observable and $H$ is the Hamiltonian operator. In the setting of deformation quantization, $[\argument {,} \argument]$ should be interpreted as the commutator of the formal star product $[A, B] = A \star B - B \star A$ and the (reduced) Planck constant $\hbar$ should be replaced by the formal deformation parameter which we usually denote by $t$. Actual quantum mechanical observables are typically represented as operators on a (pre)Hilbert space. (See \cite{waldmann} for more details.)

\paragraph{Involutions} In the classical setting, the observable algebra consists of complex-valued smooth functions and the \emph{physical observables} are the real-valued functions, i.e.\ the complex-valued functions invariant under complex conjugation, whereas in the quantum mechanical setting, the physical observables are the self-adjoint operators. Complex conjugation and taking adjoints are consolidated in the notion of a \emph{$^*$\=/involution}. Recall that a \emph{$^*$\=/algebra} $A$ is an algebra over $\mathbb C$ with a $\mathbb C$-antilinear involution $^* \colon A \to A$ which satisfies $(ab)^* = b^* a^*$ for all $a,b \in A$. An ideal $I \subset A$ is called a \emph{$^*$\=/ideal} if $I^* = I$. Examples of $^*$\=/algebras are the algebra of complex-valued polynomial functions $\Pol(\mathbb R^d)$ with complex conjugation as $^*$\=/involution or the algebra of adjointable operators on a (pre)Hilbert space.

We will follow the approach of deformation quantization and we thus work with {\it complex-valued functions}. Throughout the article we use the following notation.

\begin{notation}
\label{notations:algebras}
Consider the following function spaces:
\begin{align*}
\Smooth (X) &= \Smooth (X, \mathbb R) \otimes \mathbb C \quad \text{smooth complex-valued functions on a real manifold $X$} \\
\Ana (\mathbb R^d) &\phantom{{}={}} \text{analytic complex-valued functions on $\mathbb R^d$ with infinite radius of convergence} \\
\Pol (\mathbb R^d) &\simeq \mathbb C [x_1, \dotsc, x_d] \quad \text{complex-valued polynomial functions on $\mathbb R^d$} \\
\Hol (M) &\phantom{{}={}} \text{holomorphic functions on a complex manifold $M$.}
\end{align*}
Note that $\Pol (\mathbb R^d) \subset \Ana (\mathbb R^d) \subset \Smooth (\mathbb R^d)$ and $\Hol (\mathbb C^d) \simeq \Ana (\mathbb R^d)$ by restriction. All of these are infinite-dimensional $\mathbb C$-vector spaces which admit a commutative algebra structure given by the usual pointwise multiplication of functions and a $^*$\=/involution given by pointwise complex conjugation, or in case of $\Hol(\mathbb C^d)$ the $^*$\=/involution induced by the complex conjugation on $\Ana(\mathbb R^d)$. But by themselves we often refer to them as ``function spaces'' to emphasize that we consider them also with other non-commutative multiplications obtained as quantizations of various Poisson structures. (In \S\ref{sec:continuity} we also obtain other function spaces as completions of $\Pol (\mathbb R^d)$ with respect to various locally convex topologies.)
\end{notation}

\paragraph{Star products} We briefly recall the notion of a formal star product and refer to \cite{cattaneo,kontsevich1,waldmann} for more background. 

\begin{definition} \label{definition:formalStarProduct}
Let $(X, \eta)$ be a Poisson manifold. A {\it formal star product} (or {\it formal deformation quantization}) for $(X, \eta)$ is a $\mathbb C \llrr{t}$-bilinear multiplication
%\begin{gather*}
\[
\star \colon \Smooth (X) \llrr{t} \times \Smooth (X) \llrr{t} \to \Smooth (X) \llrr{t}
\]
given by
%\shortintertext{given by}
\[
f \star g = fg + \sum_{n \geq 1} t^n B_n (f, g)
\]
%\end{gather*}
where $B_n$ are $\mathbb C \llrr{t}$-bilinear extensions of $\mathbb C$-bilinear maps $\Smooth (X) \times \Smooth (X) \to \Smooth (X)$, satisfying the following properties:
\begin{enumerate}
\item $\star$ is associative
\item $B_1 (f, g) - B_1 (g, f) = \I \{ f, g \}_\eta$
\item $1 \star f = f = f \star 1$.
\end{enumerate}
Further, $\star$ is a {\it differential} star product if the bilinear maps $B_n$ are extensions of bidifferential operators.
\end{definition}

In other words, a formal star product is a formal one-parameter associative deformation (in the sense of Gerstenhaber \cite{gerstenhaber}) of the commutative product on the algebra $\Smooth (X)$ of classical observables, where the base of deformation is the complete local Noetherian algebra $\mathbb C \llrr{t}$ with maximal ideal $(t)$. The formal deformation parameter $t$ stands in for the (reduced) Planck constant $\hbar$ and $t$ should be evaluated to this constant where possible. However, in the generality of Definition \ref{definition:formalStarProduct} this is {\it not} possible unless the Poisson structure vanishes identically (see e.g.\ \cite[\S 6.1.1]{waldmann}) and a large part of this article is devoted to developing a framework for making sense of the evaluation $t \mapsto \hbar$, by working with combinatorial star products.

\paragraph{Fréchet algebras} Recall that a Fréchet space is a Hausdorff locally convex topological vector space which is complete and whose topology can be induced by a countable family of seminorms \cite{fragoulopoulou}. A {\it Fréchet algebra} is an associative algebra $(A, \star)$, where $A$ is a Fréchet space and the multiplication $\star \colon A \times A \to A$ is jointly continuous, i.e.\ if $f_n \to f$ and $g_n \to g$ in the Fréchet topology on $A$, then $f_n \star g_n \to f \star g$.

Fréchet spaces are a natural generalization of the notion of Banach spaces, whose topology is induced by a single norm. Normed algebras such as Banach algebras cannot contain elements $x, p$ satisfying the ``canonical commutation relations'' $[x, p] = \I \hbar$, which express Heisenberg's uncertainty principle and are at the heart of quantum mechanics. Fréchet algebras provide a natural generalization which do not impose this restriction.

All the strict deformation quantizations we obtain in this article are Fréchet algebras, obtained from completing the space $\Pol (\mathbb R^d)$ of polynomial functions, endowed with a non-commutative associative multiplication $\star_\hbar$, with respect to certain locally convex topologies.

\paragraph{Strict deformation quantization} Just as there are many approaches to quantization, there are many approaches to strict quantization \cite{bieliavskygayral,bordemannmeinrenkenschlichenmaier,landsman,natsumenestpeter,rieffel,sheu,waldmann1} (see also \cite[\S 2]{hawkins} for an overview). For the purposes of this article we shall work with the following notion.

\begin{definition}
\label{definition:strict}
Let $\eta$ be a Poisson structure on $\mathbb R^d$. By a \emph{strict deformation quantization} of $(\mathbb R^d, \eta)$ we shall mean a family $\{ A_\hbar \}_{\hbar \in [0, \epsilon)}$ of Fréchet algebras $A_\hbar = (A, \star_\hbar)$ defined on a common underlying Fréchet space $A$ satisfying
\begin{enumerate}
\item $\Pol (\mathbb R^d) \subset A_0 \subset \Smooth (\mathbb R^d)$ as commutative algebras (in particular $\star_0$ is the usual commutative product of functions)
\item the subspace $A \subset \Smooth(\mathbb R^d)$ is closed under the Poisson bracket 
\item \label{strict3} for fixed $f, g \in A$, the maps $\hbar \mapsto f \star_\hbar g$ are continuous
\item \label{strict4} for fixed $f, g \in A$, we have that $\frac{1}{\I \hbar} (f \star_\hbar g - g \star_\hbar f) \to \{ f, g \}_\eta$ as $\hbar \to 0$.
\end{enumerate}
\end{definition}

This definition should be viewed as a working definition geared towards strict quantizations of polynomial Poisson structures on $\mathbb R^d$, rather than a general notion of strict (deformation) quantization covering the various notions that appear in the literature. Let us therefore briefly motivate this definition and put it into context. Firstly, the physical value of $\hbar$ is a positive constant and the quantization should be ``well-behaved'' in the limit $\hbar \to 0$, the so-called \emph{classical limit}. We thus take a strict deformation quantization to be given by a family indexed by $[0, \epsilon)$ and require continuity of $\hbar \mapsto f \star_\hbar g$ with respect to the Fréchet topology of $A$. In other contexts one may relax this assumption and work with a family of algebras indexed by a set of real numbers which is only assumed to have $0$ as an accumulation point (see e.g.\ \cite{bordemannmeinrenkenschlichenmaier}). It turns out that our examples often satisfy an even stronger condition: $\hbar$ may be evaluated to any complex value in the closed upper half-plane $\{ \hbar \in \mathbb C \mid \operatorname{Im} (\hbar) \geq 0 \}$ and for fixed $f, g \in A$, their star product $f \star_\hbar g$ even depends holomorphically on $\hbar$ in the open upper half-plane with continuity not only on $[0, \epsilon)$ but on the whole real line (see \S\ref{subsec:convergence}).

To ensure that the algebras contain a reasonably large number of functions, we require that $A_0$ contain the algebra $\Pol (\mathbb R^d)$ of polynomial functions, although here one might also choose to replace $\Pol (\mathbb R^d)$ by any other preferred class of functions. (Note that on a general Poisson manifold, there is no reasonable notion of ``polynomial functions'', so in general one indeed has to work with other function classes.)

\begin{remark}[Strict quantizations in the C$^*$\=/algebraic setting]
Although Definition \ref{definition:strict} closely parallels other common definitions of strict quantizations, in the C$^*$\=/algebraic setting essentially all of the conditions in the definition are altered slightly (see e.g.~\cite[Def.~9.2]{rieffel}). Instead of a family of general Fréchet algebras, one considers a (continuous) ``field of C$^*$\=/algebras'' $\{ B_\hbar \}_{\hbar \in [0, \epsilon)}$ \cite{dixmier}. The individual algebras $B_\hbar$ typically arise as completions of a fixed vector space $A$ with respect to different C$^*$\=/norms, so that the completed algebras $B_\hbar$ usually do not all have the same underlying vector space. Accordingly, the continuity properties in \refitem{strict3} and \refitem{strict4} are then only imposed for elements in this fixed vector subspace. Lastly, this fixed vector subspace is never the algebra of polynomial functions, as the individual algebras $B_\hbar$ contain elements associated to bounded and continuous (but not necessarily smooth) functions, and polynomial functions are not bounded. %In particular, the individual algebras $B_\hbar$ do not contain the space of polynomial functions as vector subspace, as the elements in $B_\hbar$ are intended to model bounded operators on a Hilbert space, but polynomial functions correspond to unbounded operators.
\end{remark}

\subsection{Combinatorial star products}
\label{subsection:combinatorial}

Convergence and continuity results used in the construction of strict deformation quantizations usually arise from concrete formulae for formal star products. In order to obtain explicit formulae, we work with the combinatorial star products introduced in \cite{barmeierwang}. We now give a brief review of these star products and in \S\ref{subsec:combinatorialinvolutions} also consider variants which are compatible with $^*$\=/involutions.

The construction of combinatorial star products can be described as follows. Let
\[
A = \mathbb C \langle x_1, \dotsc, x_d \rangle / (x_j x_i - x_i x_j)_{1 \leq i < j \leq d} \simeq \mathbb C [x_1, \dotsc, x_d] \simeq \Pol (\mathbb R^d)
\]
where $\mathbb C \langle x_1, \dotsc, x_d \rangle$ is the free algebra generated by $x_1, \dotsc, x_d$ and $(x_j x_i - x_i x_j)_{1 \leq i < j \leq d}$ denotes the two-sided ideal generated by the commutativity relations. Choose the basis $\{ x^K \}_{K \in \mathbb N_0^d}$, where we use the multi-index notation $x^K = x_1^{K_1} x_2^{K_2} \dots x_d^{K_d}$ for $K = (K_1, \dotsc, K_d) \in \mathbb N_0^d$ and $\lvert K \rvert = K_1 + \dotsb + K_d$.

\begin{definition}[{\cite[Def.~7.18]{barmeierwang}}]
\label{definition:combinatorial}
Associated to any element $\widetilde \phi \in \Hom (\mathbb C \{ x_j x_i \}_{1 \leq i < j \leq d}, A) \hatotimes (t)$ viewed as a formal series $\widetilde \phi = \widetilde \phi_1 t + \widetilde \phi_2 t^2 + \dotsb$ of $\mathbb C$-linear maps, we define a {\it combinatorial star product} as the $\mathbb C \llrr{t}$-bilinear operation
\[
\star \colon A \llrr{t} \times A \llrr{t} \to A \llrr{t}
\]
where $x^K \star x^L$ is the result of recursively reordering the monomial $x^K x^L$ (viewed as an element in $\mathbb C \langle x_1, \dotsc, x_d \rangle$), starting from the right, by replacing each occurrence of $x_j x_i$ for $i < j$ by $x_i x_j + \widetilde \phi (x_j x_i)$\footnote{This recursive reordering is well-defined at all orders of $t$ since $\widetilde \phi (x_j x_i)$ is a formal power series starting in first order in $t$.}, where the terms appearing in each $\widetilde \phi_n (x_j x_i)$ for $n \geq 1$ are expressed in the form $x_1^{J_1} x_2^{J_2} \dots x_d^{J_d}$.
\end{definition}

\begin{remark}
\label{remark:reduction}
This recursive reordering can be formalized and generalized to arbitrary finitely generated algebras using the notion of a so-called {\it reduction system}, used by G.M.~Bergman to prove the Diamond Lemma \cite{bergman}. In this context, the operation of replacing $x_j x_i$ by $x_i x_j + \widetilde \phi (x_j x_i)$ is called a \emph{reduction}. Indeed, the deformation theory of any finitely generated algebra is equivalent to the deformation theory of any suitable reduction system \cite{barmeierwang}.
\end{remark}

For general $\widetilde \phi$, the operation $\star$ given in Definition \ref{definition:combinatorial} need not be associative. However, $\widetilde \phi = \widetilde \phi_1 t + \widetilde \phi_2 t^2 + \cdots$ can be viewed as a formal power series of degree $2$ elements in the Koszul cochain complex $\mathrm K^\bullet = \Hom (\mathrm K_\bullet, A)$, where
\[
\mathrm K_m = \mathbb C \{ x_{i_m} \dots x_{i_2} x_{i_1} \}_{1 \leq i_1 < \dotsb < i_m \leq d} \simeq \mathbb C^{\binom d m} \punkt
\]
Since $\mathrm K^\bullet$ has trivial differential and computes the Hochschild cohomology of $A$, its suspension $\mathrm K^{\bullet+1}$ admits a natural L$_\infty$ algebra structure $(\mathrm K^{\bullet + 1}, \langle \argument {,} \argument \rangle, \langle \argument {,} \argument {,} \argument \rangle, \dotsc)$ by viewing it as a {\it minimal model} (see e.g.\ \cite[\S 4.5.1]{kontsevich1} or \cite[\S 4]{manetti}) for the Hochschild cochain complex $(\Hom (A^{\otimes \bullet + 1}, A), d, [\argument {,} \argument])$ endowed with its DG Lie algebra structure given by the Hochschild differential and the Gerstenhaber bracket. An explicit L$_\infty$ quasi-isomorphism $(\mathrm K^{\bullet + 1}, \langle \argument {,} \argument \rangle, \langle \argument {,} \argument {,} \argument \rangle, \dotsc) \to (\Hom (A^{\otimes \bullet + 1}, A), d, [\argument {,} \argument])$ was constructed in \cite{barmeierwang} in a more general context via homotopy transfer. The resulting L$_\infty$ algebra structure on $\mathrm K^{\bullet + 1}$ and the associativity of $\star$ are related as follows.

\begin{theorem}[\cite{barmeierwang}]
\label{theorem:equivalence}
Let $A = \mathbb C [x_1, \dotsc, x_d]$ and $\widetilde \phi \in \mathrm K^2 \hatotimes (t)$. Then the following are equivalent:
\begin{enumerate}
\item $\widetilde \phi$ is a Maurer--Cartan element of $\mathrm K^{\bullet + 1} \hatotimes (t)$.
\item $\star$ is associative.
\item \label{equiv3} $(x_k \star x_j) \star x_i = x_k \star (x_j \star x_i)$ for all $1 \leq i < j < k \leq d$.
\item $\widetilde A = \mathbb C \langle x_1, \dotsc, x_d \rangle \llrr{t} / (x_j x_i - x_i x_j - \widetilde \phi (x_j x_i))_{1 \leq i < j \leq d}$ is a formal deformation of $A$. \qedhere
\end{enumerate}
\end{theorem}

Note that if the equivalent conditions in the previous theorem are satisfied, then the $\mathbb C \llrr{t}$-linear map
\begin{equation}
\label{rho}
\begin{tikzpicture}[baseline=-2.6pt,description/.style={fill=white,inner sep=2pt}]
\matrix (m) [matrix of math nodes, row sep=.25em, text height=1.5ex, column sep=0em, text depth=0.25ex, ampersand replacement=\&, column 3/.style={anchor=base west}, column 1/.style={anchor=base east}]
{
	\rho \colon \&[-.7em] \mathbb C [x_1, \dotsc, x_d] \llrr{t} \&[2em] \widetilde A \\
	\& x^K \& x^K \\
};
\path[->,line width=.4pt,font=\scriptsize]
(m-1-2) edge (m-1-3)
;
\path[|->,line width=.4pt]
(m-2-2) edge (m-2-3)
;
\end{tikzpicture}
\end{equation}
is a vector space isomorphism and $\rho$ intertwines the combinatorial star product with the product $\cdot$ of the quotient $\widetilde A$, meaning that
\begin{equation}
f \star g = \rho^{-1} (\rho (f) \cdot \rho (g)) \punkt
\end{equation}
Indeed, performing reductions $x_j x_i \mapsto x_i x_j + \widetilde \phi (x_j x_i)$ corresponds precisely to expressing the product of the quotient algebra $\widetilde A = \mathbb C \langle x_1, \dotsc, x_d \rangle
 \llrr{t} / (x_j x_i - x_i x_j - \widetilde \phi (x_j x_i))_{1 \leq i < j \leq d}$ in the basis $\{ x^K \}_{K \in \mathbb N_0^d}$. In particular, it follows that the order of reductions does not matter.

Given any $\widetilde \phi = \widetilde \phi_1 t + \widetilde \phi_2 t^2 + \dotsb \in \mathrm K^2 \hatotimes (t)$ satisfying the equivalent conditions of Theorem \ref{theorem:equivalence}, its first-order term $\widetilde \phi_1 \in \mathrm K^2$ defines a Poisson structure $\eta$ by
\[
\{ x_j, x_i \}_\eta = \frac{1}{\I} \widetilde \phi_1 (x_j x_i)
\]
for any $1 \leq i < j \leq d$, where the Jacobi identity follows from the associativity of $\star$ in order $t^2$. (The factor of $1/\I$ is a matter of convention and could be omitted, but it naturally appears when $\eta$ is to be a real Poisson structure which is quantized by the combinatorial star product $\star$ in the sense of Definition~\ref{definition:formalStarProduct}.)

Since $\mathrm K^{\bullet + 1}$ is L$_\infty$-quasi-isomorphic to $\Hom (A^{\otimes \bullet + 1}, A)$, any star product quantizing a polynomial Poisson structure $\eta$ is equivalent to a combinatorial star product for some $\widetilde \phi \in \mathrm K^2 \hatotimes (t)$. Indeed, one may construct such a combinatorial star product by defining $\widetilde \phi$ by
\begin{equation}
\label{phi}
\widetilde \phi (x_j x_i) = \I \{ x_j, x_i \}_\eta t + \mathrm O (t^2)
\end{equation}
where the higher-order terms in $t$ are chosen such that $\star$ satisfies the associativity condition on linear monomials as in Theorem \ref{theorem:equivalence} \refitem{equiv3}. The choice of higher-order terms is far from unique --- for certain Poisson structures the higher-order terms may be chosen to be zero, but as we shall see, the choice also affects the convergence and continuity properties of the resulting star product. In \S\ref{subsec:convergence} we will illustrate in several examples how to use the combinatorial star product in practice to obtain explicit formulae for star products and study their convergence properties.

A first simple example of the condition in Theorem \ref{theorem:equivalence} \refitem{equiv3} is the following.

\begin{example}[Quantization of the log-canonical Poisson structure]
\label{example:poissonStructure:logcanonicald}
Let $\eta$ be the \emph{log-canonical Poisson structure}\footnote{The name ``log-canonical'' derives from the observation that $\ln x_1, \dotsc, \ln x_d$ are ``canonical'' coordinates in the sense that $\{ \ln x_j, \ln x_i \} = 1$.} on $\mathbb R^d$, which is determined by $\{ x_j, x_i \}_\eta = x_i x_j$ for any $1 \leq i < j \leq d$. Then for any formal power series $q = 1 + \I t + \mathrm O (t^2) \in \mathbb C \llrr{t}$ we may set
\[
x_j \star x_i = 
\begin{cases}
x_j x_i & \text{if } i \geq j \\%\komma\\
q x_i x_j & \text{if } i < j \punkt
\end{cases}
\]
(In the notation of Definition \ref{definition:combinatorial}, $\star$ is the combinatorial star product associated to the element $\widetilde \phi$ given by $\widetilde \phi (x_j x_i) = (q - 1) x_i x_j$.) 
Then $\star$ satisfies the condition of Theorem \ref{theorem:equivalence} \refitem{equiv3} as
\[
x_k \star (x_j \star x_i) = q^3 x_i x_j x_k = (x_k \star x_j) \star x_i
\]
for any $1 \leq i < j < k \leq d$, where $q^3$ appears since exactly three reductions are needed to bring $x_k x_j x_i$ into standard form ($x_k$ needs to move past both $x_j$ and $x_i$, and $x_j$ past $x_i$), each reduction contributing a factor of $q$.

By Theorem \ref{theorem:equivalence} it follows that $\star$ is associative for all polynomials in $\Pol (\mathbb R^d)$ and thus quantizes $\eta$. (Indeed, $\star$ can be defined on $\Smooth (\mathbb R^d) \llrr{t}$ as in Definition \ref{definition:formalStarProduct}, see Remark \ref{remark:smooth} below.) An explicit formula for the star product of arbitrary monomials is
\begin{equation} \label{eq:locCanonicalPoissonStructure:formulaForStarInRn}
	x^K \star x^L = q^{\sum_{1 \leq i < j \leq d} K_j L_i} x^{K+L}
\end{equation}
for all multi-indices $K, L \in \mathbb N_0^d$, since a straightforward combinatorial argument shows that $\sum_{1 \leq i < j \leq d} K_j L_i$ many reductions are needed to bring $x^K x^L$ into standard form.
\end{example}

The combinatorial star product was defined via a recursive reordering reminiscent of the so-called standard ordering of differential operators. Indeed, consider $\mathbb R^{2d}$ viewed as the cotangent bundle of $\mathbb R^d$ with its canonical symplectic structure, the first $d$ coordinate functions corresponding to position and the last $d$ to momentum variables. This symplectic structure defines a constant Poisson structure and setting $\widetilde \phi (x_j x_i) = \delta_{i+d,j} t$, the associated combinatorial star product is associative and coincides with the standard-ordered Weyl product $\star^{\mathrm{std}}$, ordering position operators to the left and momentum operators to the right.

The standard-ordered Weyl product is a differential star product (cf.\ Definition \ref{definition:formalStarProduct}) which can be given by bidifferential operators associated to graphs. By a careful analysis of the reductions, one can show that {\it any} combinatorial star product can be given by a graphical formula.

\begin{theorem}[\cite{barmeierwang}]
\label{theorem:graphical}
Let $\star$ be a combinatorial star product satisfying the equivalent conditions of Theorem~\ref{theorem:equivalence}. 
Then $\star$ can be given by the graphical formula
\begin{equation} \label{eq:graphical}
f \star g = \sum_{n \geq 0} \sum_{\Gamma \in \mathfrak G_{n, 2}} C_\Gamma (f, g)
\end{equation}
where $\mathfrak G_{n,2}$ is a set of admissible graphs for the combinatorial star product and $C_\Gamma$, associated to an admissible graph $\Gamma \in \mathfrak G_{n, 2}$, is a formal power series of bidifferential operators starting in order $t^n$.
\end{theorem}

\begin{remark}
\label{remark:smooth}
The graphical formula in Theorem \ref{theorem:graphical} (see \cite[\S10]{barmeierwang} for more details) shows that the combinatorial star product is a differential star product in the sense of Definition \ref{definition:formalStarProduct}, 
since \eqref{eq:graphical} can be used to extend it to the algebra of all smooth functions $\Smooth(\mathbb R^d)$.
By a standard argument such an extension by bidifferential operators is necessarily unique if it exists:
a bidifferential operator is continuous with respect to the standard locally convex topology of $\Smooth(\mathbb R^d)$ (locally uniform convergence of all derivatives), and $\Pol (\mathbb R^d)$ is dense in $\Smooth(\mathbb R^d)$ with respect to this topology.

The graphical formula for the combinatorial star product is very similar to Kontsevich's universal quantization formula \cite[\S 2]{kontsevich1}. In fact, the admissible graphs for the combinatorial star product are precisely the Kontsevich graphs without oriented cycles, together with a linear order of the incoming edges at each vertex \cite[Prop.~10.16]{barmeierwang}. However, the formula holds without any weights and thus makes the convergence properties more accessible.
\end{remark}

\begin{remark}
The Koszul complex $\mathrm K^{\bullet + 1}$ is isomorphic to the (shifted) complex of polyvector fields (with polynomial coefficients) which admits a natural graded Lie algebra structure given by the Schouten--Nijenhuis bracket $[\argument {,} \argument]_{\mathrm{SN}}$, and Maurer--Cartan elements of $(\mathrm K^{\bullet + 1}, [\argument {,} \argument]_{\mathrm{SN}}) \hatotimes (t)$ are precisely formal Poisson structures. The binary bracket $\langle \argument {,} \argument \rangle$ of the L$_\infty$ algebra structure on $\mathrm K^{\bullet + 1}$ coincides with the Schouten--Nijenhuis bracket \cite{barmeierwang}. However, when viewed as a minimal model of $(\Hom (A^{\otimes \bullet + 1}, A), d, [\argument {,} \argument])$, the Koszul complex $\mathrm K^{\bullet + 1}$ carries nontrivial $n$-ary brackets for all $n \geq 2$.

One way of understanding the problem of quantizing a polynomial Poisson structure is to start with a Poisson structure placed in order $1$ in $t$, thus defining a Maurer--Cartan element of $(\mathrm K^{\bullet + 1}, [\argument {,} \argument]_{\mathrm{SN}}) \hatotimes (t)$, and if necessary add suitable higher-order terms to also make it a Maurer--Cartan element of the L$_\infty$ algebra $(\mathrm K^{\bullet + 1}, \mathmbox{\langle \argument {,} \argument \rangle}, \mathmbox{\langle \argument {,} \argument {,} \argument \rangle}, \dotsc) \hatotimes (t)$. In case higher-order terms are necessary (or useful for continuity estimates), the associativity of $\star$ on linear monomials as in Theorem \ref{theorem:equivalence} \refitem{equiv3} gives a convenient combinatorial criterion to establish associativity of the combinatorial star product for all smooth functions.
\end{remark}

\subsection{\texorpdfstring{Combinatorial star products and $^*$-involutions}{Combinatorial star products and *-involutions}}
\label{subsec:combinatorialinvolutions}

In \S\ref{subsec:basic} we briefly recalled the notion of a $^*$\=/involution which allows one to identify physical observables. Like the standard-ordered Weyl product, the combinatorial star product $\star$ defined in Definition~\ref{definition:combinatorial} is in general not compatible with the natural $^*$\=/involution given by complex conjugation of functions. The reason is that the linear extension of $\Pol (\mathbb R^d) \ni x^K \mapsto x_1^{K_1} \cdot \ldots \cdot x_d^{K_d} \in \mathbb C\langle x_1, \dotsc, x_{d} \rangle$ is not compatible with the $^*$\=/involutions of $\Pol (\mathbb R^d)$ respectively $\mathbb C\langle x_1, \dots, x_{d} \rangle$ obtained by extending $(x_i)^* = x_i$ in such a way that $(ab)^* = b^* a^*$ holds. Indeed,
$\big(x_1^{K_1} \cdot \ldots \cdot x_d^{K_d} \big)^* = x_d^{K_d} \cdot \ldots \cdot x_1^{K_1} \in \mathbb C\langle x_1, \dotsc, x_{d} \rangle$ 
does not agree with the image of $\big(x^K\big)^* = x^K \in \Pol(\mathbb R^d)$.

We now present two strategies for remedying this.

\subsubsection{Combinatorial star products of Wick type}
\label{subsubsection:wick}

Our first strategy is to redefine the $^*$\=/involution by setting $(x_i)^* = x_{d+1-i}$.
Of course, the $x_i$ will not be real any more (unless $d$ is odd and $i = \frac 1 2 (d+1)$) and should not be thought of as coordinate functions on $\mathbb R^d$,
which is why we will call them $w_i$ from now on.
The combinatorial star product of $w^K$ and $w^L$ is then defined just as in \S\ref{subsec:basic} by recursively reordering $w^K w^L$, 
starting from the right, replacing each occurrence of $w_j w_i$ for $i < j$ by $w_i w_j + \widetilde \phi(w_j w_i)$.
Let 
\begin{equation} \label{eq:wickDisc}
	\antidiag \coloneqq \lbrace z \in \mathbb C^d \mid z_i = \cc z_{d+1-i} \text{ for all $i=1,\dotsc,d$}\rbrace \simeq \mathbb R^d
\end{equation} 
and $w_i \colon \antidiag \to \mathbb C$, $w_i(z) = z_i$. 
Then we have indeed $(w_i)^* = w_{d+1-i}$, where ${}^*$ is now the standard complex conjugation.

If $d$ is even, we can identify $\mathbb C^{d/2}$ with $\antidiag$ via 
\[
(z_1, \dotsc, z_{d/2}) \mapsto (\cc z_1, \dotsc, \cc z_{d/2}, z_{d/2}, \dotsc, z_1)
\]
in which case $w_1, \dotsc, w_{d/2}$ become antiholomorphic coordinates and $w_{d/2+1}, \dotsc, w_d$ holomorphic coordinates. If $d$ is odd, we can identify $\mathbb C^{(d-1)/2} \oplus \mathbb R$ with $\antidiag$ via 
\[
(z_1, \dotsc, z_{(d-1)/2}, y) \mapsto (\cc z_1, \dotsc, \cc z_{(d-1)/2}, y, z_{(d-1)/2}, \dotsc, z_1)
\]
so that $w_1, \dotsc, w_{(d-1)/2}$ become antiholomorphic coordinates, $w_{(d-1)/2+2}, \dotsc, w_d$ holomorphic coordinates, and $w_{(d-1)/2+1}$ is an additional real coordinate.
With this identification, the combinatorial star product orders antiholomorphic coordinates (``creation operators'') to the left and 
holomorphic coordinates (``annihilation operators'') to the right, and we have thus obtained a combinatorial version of so-called \emph{star products of Wick type} \cite{karabegov}, \cite[\S 5.2.3]{waldmann}
(also called {\it normal-ordered} star products).

We denote by $\Pol(\antidiag)$ the unital commutative algebra generated by the functions $w_1, \dots, w_d$ on $\antidiag$.

\begin{proposition} \label{proposition:compatibilityWithStar:WickType}
Let $\widetilde \phi \in \mathrm K^2 \hatotimes (t)$ be given by $\widetilde \phi (w_j w_i) = \I \{ w_j, w_i \}_\eta t + \mathrm O (t^2)$ and assume that $\widetilde \phi$ satisfies the equivalent conditions of Theorem \ref{theorem:equivalence}. If $\widetilde \phi$ additionally satisfies
\begin{equation} \label{eq:condition:starInvolution}
\widetilde \phi (w_j w_i)^* = \widetilde \phi (w_{d-i+1} w_{d-j+1})
\end{equation}
then $\eta$ is a real Poisson structure on $\antidiag$ and $(\Pol(\antidiag) \llrr{t}, \star)$ with complex conjugation as $^*$\=/involution is a $^*$\=/algebra quantizing $\eta$.
\end{proposition}

\begin{proof}
The assignment $(w_i)^* = w_{d+1-i}$ naturally extends to $^*$\=/involutions on $\Pol(\antidiag)$ and $\mathbb C \langle w_1, \dotsc, w_{d} \rangle$, and further to $^*$\=/involutions on $\Pol(\antidiag)\llrr t$ and $\mathbb C \langle w_1, \dotsc, w_{d} \rangle \llrr{t}$ by requiring that $t^* = t$.
The linear extension of $\Pol(\antidiag) \ni w^K \mapsto w_1^{K_1} \cdot \ldots \cdot w_d^{K_d} \in \mathbb C \langle w_1, \dotsc, w_{d} \rangle$
is compatible with these $^*$\=/involutions: $\big(w_1^{K_1} \cdot \ldots \cdot w_d^{K_d}\big)^* = w_1^{K_d} \cdot \dotsc \cdot w_d^{K_1} \in \mathbb C \langle w_1, \dotsc, w_{d} \rangle$ is the image of $(w^K)^* =  w_1^{K_d} \cdot \dotsc \cdot w_d^{K_1} \in \Pol(\antidiag)$. 
It follows that if the ideal generated by $w_j w_i - w_i w_j - \widetilde \phi(w_j w_i)$ for $1 \leq i < j \leq d$ is a $^*$\=/ideal, then complex conjugation is a $^*$\=/involution for $(\Pol(\antidiag)\llrr t, \star)$. 
This is certainly the case if
\[
(w_j w_i - w_i w_j - \widetilde \phi (w_j w_i))^* = w_{d-i+1} w_{d-j+1} - w_{d-j+1} w_{d-i+1} - \widetilde \phi (w_{d-i+1} w_{d-j+1}) %\komma
\]
which is satisfied if the condition \eqref{eq:condition:starInvolution} on $\widetilde \phi$ holds. In first order in $t$, condition \eqref{eq:condition:starInvolution} reads
\[
(\I \{ w_j, w_i \}_\eta)^* = \I \{ w_{d-i+1}, w_{d-j+1} \}_\eta = \I \{ w_i^*, w_j^* \}_\eta
\]
which is precisely the condition that $\eta$ is a real Poisson structure.
\end{proof}

\begin{example}
\label{example:complexlogcanonical}
Consider the Poisson structure on $\antidiag$ defined by $\{ w_j, w_i \}_\eta = \I w_i w_j$ if $i < j$,
which is precisely the log-canonical Poisson structure of Example~\ref{example:poissonStructure:logcanonicald} with $x$'s replaced by $w$'s and an additional factor of $\I$, introduced to make the Poisson structure real. One has
\begin{equation*}
	\{ x_j, x_i \}_\eta = \begin{cases}
		x_j x_{d+1-i} & \text{if $i<d+1-j$} \\
		\frac 1 2 (x_i^2 + x_j^2) &\text{if $i = d+1-j$} \\
		x_{d+1-j} x_i & \text{if $i>d+1-j$}
	\end{cases}
\end{equation*}
for any $1 \leq i < j \leq d$, where 
$x_j \coloneqq \operatorname{Re}(w_j) \coloneqq \frac 1 2 (w_j + w_j^*)$ for $1 \leq j \leq \lceil d/2 \rceil$ and 
$x_j \coloneqq \operatorname{Im}(w_j) \coloneqq \frac 1 {2\I} (w_j - w_j^*)$ for $\lceil d/2 \rceil + 1 \leq j \leq d$
are real coordinates. Now assume that $\star$ is a combinatorial star product satisfying
\begin{equation*}
	w_j \star w_i = 
	\begin{cases}
		w_j w_i &\text{if $i \geq j$} \\
		q w_i w_j & \text{if $i < j$}
	\end{cases}
\end{equation*}
where $q = 1 - t + \mathrm O (t^2) \in \mathbb C \llrr{t}$. (Since the Poisson structure contains an extra factor of $\I$, the condition $B_1 (f, g) - B_1 (g, f) = \I \{ f, g \}$ is satisfied for this choice of $q$.) By the same reasoning as in Example~\ref{example:poissonStructure:logcanonicald},
$\star$ satisfies the condition of Theorem~\ref{theorem:equivalence} \refitem{equiv3}, and an explicit formula for $\star$ is
	\begin{equation} \label{eq:combainatorialProduct:Wick}
		w^K \star w^L = q^{\sum_{1 \leq i < j \leq d} K_j L_i} w^{K+L} \punkt
	\end{equation}
If $q \in \mathbb R\llrr t$, then the assumptions of Proposition~\ref{proposition:compatibilityWithStar:WickType} are satisfied and complex conjugation yields a $^*$\=/involution on $(\Pol(\antidiag) \llrr t, \star)$. This can also be checked explicitly using \eqref{eq:combainatorialProduct:Wick}.
\end{example}

\subsubsection{Symmetrized combinatorial star products}
\label{subsubsection:symmetrized}

Our second strategy for obtaining star products which are compatible with a $^*$\=/involution is to symmetrize the combinatorial star product $\star$, similar to the way the standard-ordered Weyl product can be symmetrized to the Moyal--Weyl product. (Symmetrized combinatorial star products were also considered by the first-named author in joint work with Zhengfang Wang.) We shall denote symmetrized star products by a six-pointed star $\varstar$ instead of the five-pointed $\star$.

Recall the definition of $\rho$ from \eqref{rho}. We have seen in \S\ref{subsection:combinatorial} that the combinatorial star product can be given by the formula $f \star g = \rho^{-1} (\rho (f) \cdot \rho (g))$ when the equivalent conditions of Theorem~\ref{theorem:equivalence} are satisfied. If this is the case, then not only $\rho$, but also the $\mathbb C \llrr{t}$-linear map
\begin{equation}
\label{sigma}
\begin{tikzpicture}[baseline=-2.6pt,description/.style={fill=white,inner sep=2pt}]
\matrix (m) [matrix of math nodes, row sep=.75em, text height=1.5ex, column sep=0em, text depth=0.25ex, ampersand replacement=\&, column 3/.style={anchor=base west}, column 1/.style={anchor=base east}]
{
	\sigma \colon \&[-.7em] \mathbb C [x_1, \dotsc, x_d] \llrr{t} \&[2em] \mathbb C \langle x_1, \dotsc, x_d \rangle \llrr{t} / (x_j x_i - x_i x_j - \widetilde \phi (x_j x_i))_{1 \leq i < j \leq d} \\
	\& x_{i_1} \dots x_{i_k} \& \displaystyle\frac{1}{k!} \sum_{s \in \mathfrak S_k} x_{i_{s(1)}} \dots x_{i_{s(k)}} \\
};
\path[->,line width=.4pt,font=\scriptsize]
(m-1-2) edge (m-1-3)
;
\path[|->,line width=.4pt]
(m-2-2) edge (m-2-3)
;
\end{tikzpicture}
\end{equation}
for $1 \leq i_1 \leq \dotsb \leq i_k \leq d$ is an isomorphism which can be viewed as a symmetrized version of $\rho$. This symmetrized isomorphism can be used to give the following definition.

\begin{definition}
\label{definition:symmetrized}
Let $\widetilde \phi \in \mathrm K^2 \hatotimes (t)$ satisfy the equivalent conditions of Theorem~\ref{theorem:equivalence}.
Define the \emph{symmetrized combinatorial star product} $\varstar \colon \Pol(\mathbb R^d)\llrr t \times \Pol(\mathbb R^d)\llrr t \to \Pol(\mathbb R^d)\llrr t$ by
\begin{equation*}
f \varstar g = \sigma^{-1} (\sigma (f) \cdot \sigma (g)) \punkt
\end{equation*}
\end{definition}

We now consider $\mathbb C\langle x_1, \dots, x_d\rangle\llrr t$ with the $^*$\=/involution obtained by extending $x_i^* = x_i$ and $t^* = t$.
Since any symmetric element $\sum_{s \in \mathfrak S_k} x_{i_{s(1)}} \dotsb x_{i_{s(k)}} \in \mathbb C\langle x_1, \dots, x_d\rangle\llrr t$
is fixed by the involution, it follows that $\sigma$ is compatible with the $^*$\=/involution if the ideal on the right-hand side of \eqref{sigma} is a $^*$\=/ideal. This proves:

\begin{proposition} \label{proposition:compatibilityWithStar:symmetric}
	Let $\widetilde \phi \in \mathrm K^2 \hatotimes (t)$ satisfy the equivalent conditions of Theorem~\ref{theorem:equivalence} and
	assume that the ideal generated by $x_j x_i - x_i x_j - \widetilde\phi(x_j x_i)$ for $1 \leq i < j \leq d$
	is a $^*$\=/ideal of $\mathbb C\langle x_1, \dots, x_d \rangle \llrr{t}$.
	Then $(\Pol(\mathbb R^d) \llrr{t}, \varstar)$ with complex conjugation as $^*$\=/involution is a $^*$\=/algebra.
\end{proposition}

The following proposition shows that the symmetrized combinatorial star product can be viewed as a generalization of the Moyal--Weyl and Gutt star products to nonlinear Poisson structures.

\begin{proposition} \label{proposition:moyalWeylAndGutt}
Let $\eta$ be a polynomial Poisson structure with only constant and linear terms and let $\widetilde \phi \in \mathrm K^2 \hatotimes (t)$ be given by $\widetilde \phi (x_j x_i) = \I \{ x_j, x_i \}_\eta t$. Then the equivalent conditions of Theorem~\ref{theorem:equivalence} are satisfied.
Moreover, if $\eta$ is constant, then the symmetrized combinatorial star product $\varstar$ coincides with the Moyal--Weyl star product $\varstar^{\mathrm W}$ and if $\eta$ is linear, then $\varstar$ coincides with the Gutt star product $\varstar^{\mathrm G}$.
\end{proposition}

\begin{proof}
The first assertion was proved in \cite[\S 10]{barmeierwang}. The second assertion follows straightforwardly from the definition of the Moyal--Weyl and Gutt star products.
\end{proof}

\begin{example} \label{example:symmetrizedStarProdForLogCanononical}
	Let $\mathbb R^d$ be endowed with the log-canonical Poisson structure (see Example~\ref{example:poissonStructure:logcanonicald}). 
	Then $\widetilde \phi(x_j x_i) = (q - 1) x_i x_j$ for $1 \leq i < j \leq d$ satisfies the equivalent conditions of Theorem~\ref{theorem:equivalence} for any $q = 1 + \I t + \mathrm O (t^2)\in \mathbb C\llrr t$,
	and the associated symmetrized combinatorial star product $\varstar$ quantizes the log-canonical Poisson structure. 
	Note that we have $(x_j x_i - q x_i x_j)^* = x_i x_j - q^* x_j x_i = -q^*(x_j x_i - (q^*)^{-1} x_i x_j)$.
	This element is in the $^*$\=/ideal generated by $x_j x_i - q x_i x_j$ if $q q^* = 1$,
	in which case $\varstar$ is compatible with the $^*$\=/involution by Proposition~\ref{proposition:compatibilityWithStar:symmetric}.
	
The formulae for the ``standard-ordered'' combinatorial star product $\star$ and the symmetrized combinatorial star product $\varstar$ differ as follows:
\begin{flalign*}
&& x_i \star x_j &=
\begin{cases}
x_i x_j & \text{if } i < j \\
x_i^2 & \text{if } i = j \\
q x_j x_i & \text{if } i > j
\end{cases}
&
x_i \varstar x_j &=
\begin{cases}
\frac{2}{1 + q} x_i x_j & \text{if } i < j \\
x_i^2 & \text{if } i = j \\
\frac{2q}{1 + q} x_j x_i & \text{if } i > j.
\end{cases}
&&
\end{flalign*}
(In the formula for $\varstar$, the fractions $\frac 1 {1+q}$ and $\frac q{1+q}$ should be expanded as formal power series in $t$.) A general formula will be given in Proposition \ref{proposition:symmetrized}.
\end{example}

\subsection{Convergence of combinatorial star products}
\label{subsec:convergence}

For constant and linear Poisson structures on $\mathbb R^d$, the Moyal--Weyl and Gutt star products converge on the space of polynomial functions, since the formal star product of two polynomial functions is a polynomial in the formal parameter and not a formal power series. We now prove convergence of combinatorial star products for a range of {\it nonlinear} polynomial Poisson structures. 

\subsubsection{Log-canonical Poisson structure}
\label{subsubsec:convergencelogcanonical}

We shall start with a detailed discussion of convergence for the quantization of the log-canonical Poisson structure which also applies to the convergence results in subsequent sections.

In Example \ref{example:poissonStructure:logcanonicald} we gave an explicit formula for the formal combinatorial star product $\star \colon \Pol (\mathbb R^d) \llrr t \times \Pol (\mathbb R^d) \llrr t \to \Pol (\mathbb R^d) \llrr t$ quantizing the log-canonical Poisson structure on $\mathbb R^d$:
\begin{equation}
\label{logcanonical1}
x^K \star x^L = q^{\sum_{1 \leq i < j \leq d} K_j L_i} x^{K+L}
\end{equation}
where $q = 1 + \I t + \mathrm O (t^2)$ could be any formal power series in $t$.

Let $\tilde q$ be a holomorphic function in $\hbar$ defined on an open set $\Omega \subset \mathbb C$ containing $0$, with power series expansion around $\hbar = 0$ of the form $q = 1 + \I t + \mathrm O (t^2)$. For any $\hbar \in \Omega$, we may consider the strict product
\begin{equation}
\label{logcanonical}
x^K \star_\hbar x^L = \tilde q (\hbar)^{\sum_{1 \leq i < j \leq d} K_j L_i} x^{K+L}
\end{equation}
which defines an associative bilinear map $\star_\hbar \colon \Pol (\mathbb R^d) \times \Pol (\mathbb R^d) \to \Pol (\mathbb R^d)$. The strict star product $f \star_\hbar g$ of two polynomials $f, g$ is thus no longer a formal power series in $t$, but simply a polynomial in $x_1, \dotsc, x_d$.

Since $\tilde q$ is holomorphic at $0$, $q$ has positive radius of convergence $r$. Therefore, for all $\hbar \in \complexDisc = \{ z \in \mathbb C \mid \abs z < r \}$, \eqref{logcanonical} can be viewed as the evaluation of \eqref{logcanonical1} for $t \mapsto \hbar$. By the uniqueness of analytic continuation, one can recover \eqref{logcanonical} for all $\hbar \in \Omega$. Generally, for any two polynomials $f, g$, the coefficient functions of the formal star product $f \star g$ lie in a finite-dimensional subspace of $\Pol(\mathbb R^d)$. Since $f \star g$ converges on $\complexDisc$, analytic continuation implies that $f \star g$ can be evaluated to $f \star_\hbar g$ for any $f, g \in \Pol (\mathbb R^d)$ and any $\hbar \in \Omega$. Moreover, since evaluation at $\hbar \in \complexDisc$ intertwines $\star$ and $\star_\hbar$ and since the formal star product is associative, $(f \star_\hbar g) \star_\hbar h$ and $f \star_\hbar (g \star_\hbar h)$ must coincide on $\complexDisc$ and therefore on all of $\Omega$ for any fixed $f,g,h \in \Pol(\mathbb R^d)$. This gives an abstract argument showing the associativity of $\star_\hbar$ which also applies to the examples in subsequent sections.

Some natural choices for $\tilde q$ would be $\tilde q (\hbar) = 1 + \I \hbar$, $\E^{\I \hbar}$, or $\frac{1}{1 - \I \hbar}$ whose power series expansions are all of the form $q = 1 + \I t + \mathrm O (t^2)$. The formal star product $\star$, and in particular the Poisson structure it quantizes, can be recovered from the family of strict star products $\star_\hbar$: For fixed $f, g \in \Pol (\mathbb R^d)$, the map $\hbar \mapsto f \star_\hbar g$ defines a holomorphic function $\Omega \to \Pol (\mathbb R^d)$ whose power series expansion around $\hbar = 0$ is the formal combinatorial star product $f \star g$. (Note that for fixed polynomials $f, g$, their star product $f \star_\hbar g$ lies in a finite-dimensional subspace of $\Pol (\mathbb R^d)$, so here the notion of holomorphy is the standard one.) We shall no longer distinguish a holomorphic function $\tilde q$ and its power series expansion $q$ around $\hbar = 0$.

In \S\ref{sec:continuity} we show how to extend the strict star products $\star_\hbar$ to larger function spaces which are obtained as completions with respect to some locally convex topology, where the details will depend on whether $\abs {q (\hbar)} \leq 1$ or not. Note that this final condition is for example satisfied for $q (\hbar) = \mathrm e^{\I \hbar}$ whenever $\hbar$ lies in the closed upper half-plane, in particular for all real values of $\hbar$.

\begin{remark}\label{remark:compatibilityWithStar:Wick:strict}
The above discussion on convergence applies mutatis mutandis to the combinatorial star products of Wick type, in particular to the quantization of the Poisson structure $\{ w_j , w_i \} = \I w_i w_j$ where $1 \leq i < j \leq d$ on $\mathbb R^d$ introduced in Example \ref{example:complexlogcanonical} whose formula in the $w_i$ coordinates is very similar to the formula for the log-canonical Poisson structure in the $x_i$ coordinates, 
even though it quantizes a different real Poisson structure. 
Note, however, that $q$ must now be of the form $q = 1 - t + \mathrm O (t^2)$. 
Choosing $q (\hbar) = \mathrm e^{-\hbar}$ we have that $\abs {q (\hbar)} \leq 1$ whenever $\hbar$ lies in the closed right half-plane, in particular for all non-negative real values of $\hbar$. Like the formal star product $\star$, the star product $\star_\hbar$ obtained by evaluating $t \mapsto \hbar$ is compatible with the $^*$\=/involution if $q(\hbar)$ is real.
\end{remark}

\subsubsection{Other polynomial Poisson structures}
\label{subsubsection:general}

We now consider other examples of polynomial Poisson structures which can be of arbitrary polynomial degree and find quantizations using the criterion
\begin{equation}
\label{overlap}
(x_k \star x_j) \star x_i = x_k \star (x_j \star x_i)
\end{equation}
for $1 \leq i < j < k \leq d$ given in Theorem \ref{theorem:equivalence}, where $\star$ is the combinatorial star product associated to some element $\widetilde \phi \in \mathrm K^2 \hatotimes (t)$.

Solving \eqref{overlap} for the log-canonical Poisson structure was particularly straightforward (see Example \ref{example:poissonStructure:logcanonicald}). One might expect that for more general polynomial Poisson structures quantizations are very difficult to construct by hand. However, it turns out that simply choosing
\begin{equation}
\label{naive}
\widetilde \phi (x_j x_i) = \I \{ x_j, x_i \}_\eta t
\end{equation}
is ``often'' already enough to satisfy \eqref{overlap}. (The reader may check that this is true for about half of the classes of quadratic Poisson structures on $\mathbb R^3$ in the classification by J.-P.~Dufour and A.~Haraki \cite{dufourharaki}.) An example of such a computation is the following.

\begin{example}
\label{example:poissonStructure:nonQuadratic}
Let $N \in \mathbb N_0$ be fixed and consider the Poisson structure
\[
\eta = (y z + x^N) \frac{\partial}{\partial z} \wedge \frac{\partial}{\partial y}
\]
on $\mathbb R^3$. For $N = 2$, this is the Poisson structure given by D.~Manchon, M.~Masmoudi, and A.~Roux \cite{manchonmasmoudiroux} as an example of a quadratic Poisson structure that cannot be quantized via Drinfel'd twists.

Let $\widetilde \phi$ be given as in \eqref{naive}, i.e.\ $\widetilde \phi (y x) = \widetilde \phi (z x) = 0$ and $\widetilde \phi (z y) = \I \{ z, y \}_\eta t = (y z + x^N) \I t$. The associated combinatorial star product $\star$ satisfies \eqref{overlap} since
\begin{align*}
z \star (y \star x) &= \eqmakebox[xyz][c]{$z \star (x y) = \rho^{-1} (x \cdot z \cdot y)$} = x y z (1 + \I t) + x^{N + 1} \I t \\
(z \star y) \star x &= \eqmakebox[xyz][c]{$\bigl( y z (1 + \I t) + x^N \I t \bigr) \star x$} = x y z (1 + \I t) + x^{N + 1} \I t
\end{align*}
and thus $\star$ quantizes $\eta$.
\end{example}

The next example is a slight modification of Example \ref{example:poissonStructure:nonQuadratic} which illustrates that even when the combinatorial star product associated to the naive choice \eqref{naive} does \emph{not} satisfy \eqref{overlap}, it is often easy enough to find suitable higher correction terms.

\begin{example}
\label{example:poissonStructure:nonQuadratic:exact}
Let $N \in \mathbb N_0$ be fixed and consider the exact Poisson structure on $\mathbb R^3$ associated to the polynomial function $- x y z - \frac{1}{N+1} x^{N+1}$, i.e.\
\begin{equation*} \label{eq:poissonStructure:nonQuadratic:exact}
\eta = x y \frac{\partial}{\partial y} \wedge \frac{\partial}{\partial x} - x z \frac{\partial}{\partial z} \wedge \frac{\partial}{\partial x} + (y z + x^N) \frac{\partial}{\partial z} \wedge \frac{\partial}{\partial y} \punkt
\end{equation*}
Let $p,q, r, s \in \mathbb C \llrr{t}$ be formal power series of the form $p = 1 + \I t + \mathrm O (t^2)$, 
$q = 1 + \I t + \mathrm O (t^2)$, $r = 1 + \I t + \mathrm O (t^2)$, and $s = 1 - \I t + \mathrm O (t^2)$ and let $\widetilde \phi$ be determined by
\begin{align*}
\widetilde \phi (y x) &= (r - 1) x y \\
\widetilde \phi (z x) &= (s - 1) x z \\
\widetilde \phi (z y) &= (q - 1) yz + (p-1) x^N
\end{align*}
which is essentially \eqref{naive} with general yet-to-be-determined higher-order terms. Computing the associated combinatorial star product one finds
\begin{align*}
z \star (y \star x) &= z \star (r x y) = \rho^{-1}(rs x \cdot z \cdot y) = q r s xyz + (p-1) rs x^{N+1} \\
(z \star y) \star x &= (q yz + (p-1) x^N) \star x = q r s x y z + (p-1) x^{N+1}
\end{align*}
and comparing the right-hand sides we see that it suffices to set $s = r^{-1}$ to satisfy \eqref{overlap} in which case $\star$ quantizes $\eta$.
\end{example}

The following proposition gives explicit formulae and convergence results for the star products in Examples \ref{example:poissonStructure:nonQuadratic} and \ref{example:poissonStructure:nonQuadratic:exact}. For a word $w = (w_1, \dots, w_k) \in \set{0,1}^k$ in $0$'s and $1$'s we write $\abs w \coloneqq \sum_{i=1}^k w_i$ and 
$w_{1\dots i}=(w_1, \dots, w_i) \in \set{0,1}^i$, where $i \in \{1, \dots, k\}$. We use the convention that $\set{0,1}^0$ contains one element, the empty word $\emptyset$ with $\abs \emptyset = 0$.

\begin{proposition}
\label{proposition:nonQuadratic}
\hyperdef{equation}{\theequation}{} %to fix error with link
\begin{enumerate}
\item \label{item:proposition:nonQuadraticFormula:i}
For $N \in \mathbb N_0$ and formal power series $p = 1 + \I t p_1 + \mathrm O(t^2)$, $q = 1 + \I t q_1 + \mathrm O (t^2)$,
and $r = 1 + \I t r_1 + \mathrm O (t^2) \in \mathbb C \llrr t$, the combinatorial star product $\star$ determined by 
\begin{align*} \label{eq:nonQuadratic:product}
	\widetilde \phi(y x) &= (r - 1) x y \\
	\widetilde \phi(z x) &= (r^{-1} - 1) x z\\
	\widetilde \phi(z y) &= (q - 1) y z + (p-1) x^N
\end{align*}
is a formal deformation quantization of the Poisson structure with Poisson bivector
\begin{equation} \label{eq:nonQuadratic:poissonStructure}
	\eta =  r_1 xy \frac{\partial}{\partial y} \wedge \frac{\partial}{\partial x}- r_1 xz \frac{\partial}{\partial z} \wedge \frac{\partial}{\partial x}  + (q_1 yz + p_1 x^N) \frac{\partial}{\partial z} \wedge \frac{\partial}{\partial y}\punkt 
\end{equation}
\item \label{item:proposition:nonQuadraticFormula:ii}
For any $i, j, k, \ell, m, n \in \mathbb N_0$, one has
\begin{equation}
	\label{eq:productformula:nonQuadratic:exact}
	x^i y^j z^k \star x^\ell y^m z^n
	= \sum_{\substack{w \in \set{0,1}^k \\ \lvert w \rvert \leq m}} r^{(j-k)\ell + j N \abs w} \lambda_m(w) x^{i+\ell + N \abs w} y^{j+m - \abs w} z^{k+n -\abs w} %\komma
\end{equation}
where $\lambda_m(\emptyset) = 1$ and $\lambda_m(w) \coloneqq \prod_{i=1}^k r^{-N \abs{w_{1\dots i-1}}} \tilde \lambda_{m}(w_{1\dots i-1}, w_i)$ for any $w \in \set{0,1}^k$, $k \in \mathbb N$ with
\begin{equation} \label{eq:lambda}
	\tilde \lambda_{m}(w, s) = \begin{cases}
		q^{m-\abs{w}}  &\text{if $s = 0$} \\
		(p-1) \sum_{j=0}^{m-\abs w-1} (q r^N)^j &\text{if $s = 1$} \punkt
	\end{cases}
\end{equation}
\item \label{item:proposition:nonQuadraticFormula:iii}
Let $\Omega \subset \mathbb C$ be an open connected neighbourhood of $0$ and $p$, $q$ and $r$ be power series expansions of holomorphic functions on $\Omega$, without zeros in the case of $r$. Then $\star$ can be evaluated to an associative product $\star_\hbar \colon \Pol(\mathbb R^3) \times \Pol(\mathbb R^3) \to \Pol(\mathbb R^3)$
for any $\hbar \in \Omega$.
\item \label{item:proposition:nonQuadraticFormula:iv}
For any fixed $f, g \in \Pol(\mathbb R^3)$, the map $\Omega \to \Pol(\mathbb R^3)$, $\hbar \mapsto f \star_\hbar g$ takes values in a finite-dimensional subspace $V \subset \Pol(\mathbb R^3)$ and depends holomorphically on $\hbar$. Moreover, $\frac{1}{\I\hbar}(f \star_\hbar g - g \star_\hbar f) \to \{f,g\}_\eta$ as $\hbar \to 0$ (with respect to the usual topology on the finite-dimensional space $V$). \qedhere
\end{enumerate}
\end{proposition}

\begin{proof} The same computation as in Example~\ref{example:poissonStructure:nonQuadratic:exact} shows that $\widetilde\phi$ satisfies condition \refitem{equiv3} in Theorem~\ref{theorem:equivalence} and it is straightforward to verify that the corresponding star product quantizes the Poisson structure \eqref{eq:nonQuadratic:poissonStructure}, thus showing part \refitem{item:proposition:nonQuadraticFormula:i}.
	Part \refitem{item:proposition:nonQuadraticFormula:ii} can be shown by a computation in the algebra 
	$\widetilde A = \mathbb C\langle x,y,z\rangle \llrr{t} / (y x - x y - \widetilde \phi(y x), z x - x z - \widetilde \phi(z x), z y - y z - \widetilde \phi(z y))$, where we shall omit the product sign ${}\cdot{}$ for brevity.
	By induction over $m$, one can easily show that
	\begin{equation*}
		z y^m = q^m y^m z
		+ (p-1) \sum_{j=0}^{m-1} \big(q r^N\big)^j x^N y^{m-1} 
	\end{equation*}
	holds for all $m \in \mathbb N_0$. 
	Next, we prove that $z^k y^m = \sum_{w \in \set{0,1}^k, \lvert w \rvert \leq m} \lambda_m(w) x^{N \abs w} y^{m - \abs w} z^{k -\abs w}$ holds for all $k,m \in \mathbb N_0$ by induction over $k$:
	If $k = 0$, this is clearly true, and if this holds for some $k \in \mathbb N_0$, then 
	\begin{align*}
		z^{k+1} y^{m} 
		&= 
		\sum_{\substack{w \in \set{0,1}^k \\ \lvert w \rvert \leq m}} \lambda_m(w) z x^{N \abs w} y^{m - \abs w} z^{k - \abs w}  \\
		&= \sum_{\substack{w \in \set{0,1}^k \\ \lvert w \rvert \leq m}} \lambda_m(w) r^{-N \abs w} q^{m - \abs w} x^{N \abs w} y^{m - \abs w} z^{k +1- \abs w} \\
		&\phantom{XXX}+
		\sum_{\substack{w \in \set{0,1}^k \\ \lvert w \rvert \leq m-1}} \lambda_m(w) r^{-N \abs w} (p-1) \sum_{j=0}^{m-\abs w-1} (q r^N)^j x^{N (\abs w +1)}  y^{m-1 - \abs w} z^{k - \abs w}  \\
		&=\sum_{\substack{w \in \set{0,1}^{k+1} \\ \lvert w \rvert \leq m}} \lambda_m(w) x^{N \abs w} y^{m - \abs w} z^{k + 1 - \abs w} \,.
	\end{align*}
	Using this, we finally obtain
	\begin{align*}
		x^i y^j z^k x^\ell y^m z^n 
		&= r^{(j-k)\ell} x^{i+\ell} y^j z^k y^m z^n 
		\\&= \sum_{\substack{w \in \set{0,1}^k \\ \lvert w \rvert \leq m}} r^{(j-k)\ell} \lambda_m(w) x^{i+\ell} y^j x^{N \abs w} y^{m-\abs w} z^{k-\abs w} z^n 
		\\&= \sum_{\substack{w \in \set{0,1}^k \\ \lvert w \rvert \leq m}} r^{(j-k)\ell + j N \abs w} \lambda_m(w) x^{i+\ell + N \abs w} y^{j+m-\abs w} z^{k+n-\abs w} \punkt 
	\end{align*}
	To show \refitem{item:proposition:nonQuadraticFormula:iii}, 
	note that holomorphic functions on $\Omega$ are uniquely determined by their power series expansion around $0$, 
	hence $\star$ can indeed be evaluated at any $\hbar \in \Omega$ and is associative (by analytic continuation as in \S\ref{subsubsec:convergencelogcanonical}).
	Part \refitem{item:proposition:nonQuadraticFormula:iv} follows easily from the formula obtained in part \refitem{item:proposition:nonQuadraticFormula:ii}.
\end{proof}

	Formula \eqref{eq:lambda} simplifies if one chooses $p$, $q$, and $r$ such that $p-1$ is a multiple of $qr^N - 1$.
	If $q_1 = p_1 = 1$ and $r_1 = 0$, then the star product in the previous proposition quantizes the Poisson structure from Example~\ref{example:poissonStructure:nonQuadratic}, if $q_1 = p_1 = r_1 = 1$ then it quantizes the Poisson structure from Example~\ref{example:poissonStructure:nonQuadratic:exact}. 

In \S\ref{subsubsec:continuity:combinatorial:nonquadratic} we obtain continuity estimates for the star products introduced in Proposition \ref{proposition:nonQuadratic}.

\subsubsection{Symmetrized combinatorial star products} \label{subsubsec:convergence:symmetrized}

In order to give an explicit formula for the symmetrized combinatorial star product quantizing the log-canonical Poisson structure (see Example~\ref{example:symmetrizedStarProdForLogCanononical}),
recall the definition of the \emph{$q$-multinomial coefficients}
\begin{equation} \label{eq:qmultinomial}
	\binom {\abs K} {K}_q = \frac{[\abs K]_q!}{[K_1]_q! \dots [K_d]_q! }
\end{equation}
where $K \in \mathbb N_0^d$ is a multi-index, and $[\argument]_q!$ are the {$q$-factorials} defined by $[0]_q! = 1$ and
\begin{equation*}
	[k]_q! = [k]_q [k-1]_q \dots [1]_q 
	\quad\text{with}\quad
	[k]_q = 1 + q + \dots + q^{k-1}
\end{equation*}
for any $k \in \mathbb N$.
The $q$-multinomial coefficients compute a weighted sum of all possible ways to form words with certain letters. 
More precisely, $\smash{\binom {\abs K} {K}_q} = \sum_{w} q^{\mathrm{inv}(w)}$ 
where the sum is over all words $w$ with $K_1$ many $1$'s, $K_2$ many $2$'s, and so on,  
and $\mathrm{inv}(w)$ denotes the minimum number of inversions (exchanging two consecutive letters) 
needed to change the word $1^{K_1} \dots d^{K_d}$ to $w$ \cite[Prop.~1.3.17]{stanley}.
In particular, $\binom{\abs K}{K}_q$ is a polynomial in $q$ with non-negative integer coefficients and constant term $1$.

\begin{proposition}
\label{proposition:symmetrized}
	Let $\mathbb R^d$ be endowed with the log-canonical Poisson structure and assume that $q = 1 + \I t + \mathrm O (t^2)$. 
	The symmetrized combinatorial star product from 
	Example~\ref{example:symmetrizedStarProdForLogCanononical} is given by the formula
	\begin{equation} \label{eq:symmetrizedProduct}
		x^K \varstar x^L = 
		\frac{\binom{\abs{K+L}}{K+L}}{\binom{\abs K}{K} \binom{\abs L}{L}}
		\frac{\binom{\abs K}{K}_{q} \binom{\abs 
				L}{L}_{q}}{\binom{\abs{K+L}}{K+L}_{q}}
		q^{\sum_{1 \leq i < j \leq d} K_j L_i} x^{K+L} \punkt
	\end{equation}
\end{proposition}

\begin{proof}
	We obtain
	$\binom {\abs K}{K}_{q} \rho(x^K) = \binom {\abs K}{K} \sigma(x^K)$
	from the combinatorial interpretation of the $q$-multinomial coefficients.
	Consequently,
	\begin{flalign*}
		&&x^K \varstar x^L &= \sigma^{-1} (\sigma(x^K) \sigma(x^L)) 
		\\&& &= 
		\frac{\binom{\abs K}{K}_{q} \binom{\abs L}{L}_{q}}{\binom{\abs K}{K} \binom{\abs 
				L}{L}} \sigma^{-1} (\rho(x^K) \rho(x^L)) &&
		\\&& &= 
		\frac{\binom{\abs K}{K}_{q} \binom{\abs L}{L}_{q}}{\binom{\abs K}{K} \binom{\abs L}{L}} q^{\sum_{1 \leq i < j \leq d} K_j L_i} \sigma^{-1} (\rho(x^{K+L})) &&
		\\&& &= 
		\frac{\binom{\abs K}{K}_{q} \binom{\abs L}{L}_{q} \binom{\abs{K+L}}{K+L}}{\binom{\abs K}{K} \binom{\abs L}{L} \binom{\abs{K+L}}{K+L}_{q}} q^{\sum_{1 \leq i < j \leq d} K_j L_i} x^{K+L} \punkt && \qedhere
	\end{flalign*}
\end{proof}

Assume that $q$ is the formal power series expansion of a holomorphic function $q \in \Hol(\Omega)$.
The $q$-multinomial coefficients are polynomials in $q$ which can only vanish at roots of unity (but not at $1$). It follows that the coefficients in \eqref{eq:symmetrizedProduct} are meromorphic functions on $\Omega$,
with poles only at $P \coloneqq \lbrace \hbar \in \Omega \mid q(\hbar) \neq 1$ and $q(\hbar)^n = 1$ for some $n \in \mathbb N \rbrace$.
Thus we can evaluate $\varstar$ to a product $\varstar_\hbar \colon \Pol(\mathbb R^d) \times \Pol(\mathbb R^d) \to \Pol(\mathbb R^d)$
whenever $\hbar \in \Omega \setminus P$.
Note that this is different from the situation in the previous sections, where the star products could be evaluated for all $\hbar \in \Omega$.

\begin{remark}
	The proof of Proposition~\ref{proposition:symmetrized} shows that the map $T \colon \Pol(\mathbb R^d) \to \Pol(\mathbb R^d)$,
	obtained by extending $x^K \mapsto \binom{\abs K}{K}{}_{q} \binom{\abs K}{K}{}^{-1} x^K$ linearly,
	is an equivalence transformation between the standard and symmetrized combinatorial star products,
	in the sense that $f \varstar g = T^{-1}(T f \star T g)$. 
	We will see in \S\ref{sec:continuity} that although the standard and symmetrized combinatorial star products are equivalent star products quantizing the same Poisson structure, they have different continuity properties for $\abs{q(\hbar)} > 1$, as the equivalence transformation $T$ fails to be continuous for the relevant topologies on $\Pol(\mathbb R^d)$.
\end{remark}

\section{Continuity} \label{sec:continuity}

In this section we extend the combinatorial star products introduced in \S\ref{sec:starproducts} to bigger function spaces, using the observation that a continuous bilinear map of locally convex topological vector spaces extends to the completion (see Lemma \ref{lemma:continuous}). Using convergence results of combinatorial star products (cf.\ \S\ref{subsec:convergence}) we may evaluate the combinatorial star product $\star$ to $\star_\hbar$ for some complex value of $\hbar$ and then proceed to find locally convex topologies on $\Pol(\mathbb R^d)$ with respect to which $\star_\hbar$ is continuous. Then $\star_\hbar$ will extend to the completion, so that the objective is to find topologies for which the completion is as large as possible.

This approach has been used by S.~Waldmann and collaborators in several examples for Poisson structures which are of constant or linear type and in \S\ref{subsec:continuity:weylandgutt} we review the definition of the so-called \emph{\TRtop} used for this purpose. This topology depends on a parameter $R \geq 0$ and the Moyal--Weyl and Gutt star products, which quantize constant and linear Poisson structures respectively, are continuous with respect to the {\TRtop} only for certain ranges of $R$.

The subsequent sections \S\ref{subsec:continuity:combinatorial:general}--\S\ref{subsec:continuity:wick} contain the main results of this article: we find locally convex topologies for which the combinatorial star products quantizing nonlinear Poisson structures are continuous. In \S\ref{subsec:continuity:wick} we also show the existence of continuous positive linear functionals for the combinatorial star products of Wick type from Example~\ref{example:complexlogcanonical}, which allows us to represent the strict deformation quantizations faithfully on a pre-Hilbert space by means of the GNS-construction (Theorem \ref{theorem:GNS}).

It might be surprising that for nonlinear Poisson structures there are often combinatorial star products with better continuity properties than the Moyal--Weyl or Gutt star products, meaning that they extend to larger function spaces. We show the following:
\begin{enumerate}
	\item Strict star products $\star_\hbar \colon \Pol(\mathbb R^d) \times \Pol(\mathbb R^d) \to \Pol(\mathbb R^d)$ for which the product of two homogeneous polynomials of degrees $d_1$ and $d_2$ is a sum of homogeneous polynomials of degrees $\geq d_1 + d_2$ can be extended to formal power series in the coordinates (Proposition~\ref{proposition:continuity:formal}). Such products can often be obtained by quantizing polynomial Poisson structures with vanishing constant and linear components.
	
	\item \label{differences2} There are combinatorial star products, for example those quantizing homogeneous quadratic Poisson structures, which are either continuous with respect to the {\TRtop} for all $R \geq 0$ or for no $R$.
	In the latter case, we need to find finer topologies with respect to which the combinatorial star product is continuous.
	One such topology that works in great generality is introduced in \S\ref{subsec:continuity:combinatorial:general}.
	It has, however, a relatively small completion.
	\item While the star products in all examples discussed in \S\ref{sec:starproducts} can be extended to the completion in \eqref{differences2} one can typically extend them to much larger algebras when the parameters in the construction, like the formal power series $q$,
	are suitably chosen. We introduce the corresponding locally convex topologies in \S\ref{subsec:continuity:combinatorial:examples}
	and discuss the examples of \S\ref{sec:starproducts} in detail.
\end{enumerate}

Throughout we use the following standard result.

\begin{lemma}
	\label{lemma:continuous}
	Let $U, V, W$ be locally convex topological vector spaces.
	A bilinear map $\mu \colon U \times V \to W$ is continuous if and only if for every continuous seminorm $\norm{\argument}_\alpha$ on $W$ there are continuous seminorms $\norm{\argument}_\beta$ on $U$ and $\norm{\argument}_\gamma$ on $V$ such that
	\begin{equation*}
		\norm{\mu (u, v)}_\alpha \leq \norm{u}_\beta \norm{v}_\gamma \punkt
	\end{equation*}
	If $\mu$ is continuous, then $\mu$ extends to a unique continuous bilinear map $\hat \mu \colon \hat U \times \hat V \to \hat W$,
	where $\hat U$, $\hat V$, and $\hat W$ denote the completions of $U$, $V$, and $W$, respectively.
\end{lemma}

We typically apply this lemma with $U = V = W = \Pol(\mathbb R^d)$ and $\mu = \star_\hbar$, where $\star_\hbar$ is obtained from some combinatorial star product by evaluating $t \mapsto \hbar$ as in \S\ref{subsec:convergence}. We usually abuse notation and write $\mu = \star_\hbar$ instead of $\hat \mu = \hat{\star}_\hbar$ also for the extended map.

\begin{convention} \label{convention}
	Throughout this section, $\Omega \subset \mathbb C$ shall denote an open connected neighbourhood of $0 \in \mathbb C$ 
	and $\Hol (\Omega)$ the algebra of holomorphic functions on $\Omega$. We often tacitly pass from a holomorphic function $q \in \Hol (\Omega)$ depending on a complex variable $\hbar \in \Omega$ to its power series expansion around $\hbar = 0$ (which determines $q$ uniquely) with formal variable $t$.
	
For a star product $\star_\hbar \colon V \times V \to V$ quantizing a Poisson structure $\eta$, where $\hbar$ ranges over $\Omega$, 
and fixed $f, g \in V$ we define maps $\mu_{f,g}, \Delta_{f,g} \colon \Omega \to V$ by
	\begin{equation} \label{eq:productAndCommutatorMaps}
		\mu_{f,g}(\hbar) \coloneqq f \star_\hbar g
	\qquad\text{and}\qquad
	\Delta_{f,g}(\hbar) \coloneqq \begin{cases}
		\frac 1 {\I\hbar} (f \star_\hbar g - g \star_\hbar f) &\text{ if $\hbar \neq 0$} \\
		\lbrace f, g \rbrace_\eta &\text{ if $\hbar = 0$} \punkt
		\end{cases}
\end{equation}
Typically $V$ is some completion of $\Pol(\mathbb R^d)$.
\end{convention}

Our definition of strict deformation quantization (Definition \ref{definition:strict}) requires these maps to be continuous. However, we often show that $\mu_{f,g}$ and $\Delta_{f,g}$ are holomorphic, where a map $f \colon \Omega \to V$ into a Fréchet space $V$
is called \emph{holomorphic}\footnote{The definition of holomorphy through linear functionals is sometimes referred to as ``weak holomorphy''. 
	For Fr\'echet spaces, it coincides with the notion of ``strong holomorphy'' defined through the existence of the limit of the usual difference quotient.} if $\origphi \circ f \colon \Omega \to \mathbb C$ is holomorphic for all continuous linear functionals $\origphi \in V^*$.
%(This is sometimes called weakly holomorphic and coincides with another definition of strong holomorphy.) 

\subsection{Continuity of the Moyal--Weyl and Gutt star products} \label{subsec:continuity:weylandgutt}

In \cite{espositostaporwaldmann,schoetzwaldmann,waldmann1} continuity estimates for the Moyal--Weyl and Gutt star products were obtained (see also \cite{waldmann2} for a review).
In these articles, the relevant topology is the so-called \TRtop, because it depends on a real parameter $R$ and is obtained by viewing $\Pol (V) \simeq \mathrm S (V^*) \subset \TR (V^*)$ as a subspace, where $\TR (V^*)$ is a certain locally convex topology on the tensor algebra $\mathrm T (V^*) = \bigoplus_{n \geq 0} (V^*)^{\otimes n}$ on the dual of a vector space $V$. We work with the case when $V \simeq \mathbb R^d$ is finite-dimensional, in which case the {\TRtop} restricts to the usual Euclidean topology on the tensor powers $(V^*)^{\otimes n} \subset \mathrm T (V^*)$, and the parameter $R$ allows one to fine-tune the topology on $\Pol (\mathbb R^d)$ according to the Poisson structure at hand.

\begin{definition}[{\cite[Def.~3.5]{waldmann1}}]
\label{definition:TR}
	Let $R \in [0, \infty)$.
	We write $\Pol_{\TR}(\mathbb R^d)$ for the locally convex vector space obtained by endowing $\Pol(\mathbb R^d)$ with the {\TRtop}, defined by the family of norms 
	$\lbrace \norm \argument^{\TR}_C \mid C \in (0,\infty) \rbrace$, where
	\begin{equation}
		\norm\argument_C^{\TR} \colon \Pol(\mathbb R^d) \to [0, \infty) \komma
		\qquad
		\norm[\Big]{\sum_{L \in \mathbb N_0^d} f_L x^L}_{C}^{\TR} \coloneqq \sum_{L \in \mathbb N_0^d} \abs{f_L} C^{\abs L} \cdot (\abs L!)^R %\komma
	\end{equation}		
	and we write $\smash{\widehat{\Pol}_{\TR} (\mathbb R^d)}$ for its completion.
\end{definition}

Note that the {\TRtop} is coarser than the $\TR[R']$\=/topology whenever $R \leq R'$. For the completions we thus have $\widehat{\Pol}_{\TR[R']}(\mathbb R^d) \subset \widehat{\Pol}_{\TR}(\mathbb R^d)$. 

The following two theorems summarize the continuity results for the strict standard-ordered Weyl product $\star^{\mathrm{std}}_\hbar$, the strict Moyal--Weyl star product $\varstar^{\mathrm W}_\hbar$, and the strict Gutt star product $\varstar^{\mathrm G}_\hbar$, obtained from respectively $\star^{\mathrm{std}}$ (introduced before Theorem~\ref{theorem:graphical}), $\varstar^{\mathrm W}$, and $\varstar^{\mathrm G}$ (see Proposition~\ref{proposition:moyalWeylAndGutt}) by evaluating $t \mapsto \hbar \in \mathbb C$.

\begin{theorem}[\cite{waldmann1}]
	 \label{theorem:continuity:weyl}
	Let $R \geq \frac 1 2$ and $\hbar \in \mathbb C$.
	Then the standard-ordered Weyl and the Moyal--Weyl products $\star^{\mathrm{std}}_\hbar, \varstar^{\mathrm W}_{\hbar} \colon \Pol_{\TR}(\mathbb R^d) \times \Pol_{\TR}(\mathbb R^d) \to \Pol_{\TR}(\mathbb R^d)$ are continuous and therefore extend uniquely to continuous products
	$\star^{\mathrm{std}}_{\hbar}, \varstar^{\mathrm W}_{\hbar} \colon \widehat{\Pol}_{\TR}(\mathbb R^d) \times \widehat{\Pol}_{\TR}(\mathbb R^d) \to \widehat{\Pol}_{\TR}(\mathbb R^d)$. Moreover, for all $f, g \in \widehat{\Pol}_{\TR}(\mathbb R^d)$, the maps
	$\mu_{f,g}, \Delta_{f,g} \colon \mathbb C \to \widehat{\Pol}_{\TR}(\mathbb R^d)$, defined as in \eqref{eq:productAndCommutatorMaps}, are holomorphic.
\end{theorem}

Note that in \cite{schoetzwaldmann,waldmann1} generalizations to the infinite-dimensional setting are also proved, but the above version of the theorem is sufficient for our purposes and least technical to state. The following theorem gives an analogous result for linear Poisson structures, which can also be generalized to infinite dimensions if one imposes a certain technical condition on the Lie algebra.

\begin{theorem}[\cite{espositostaporwaldmann}] %[Continuity of the Gutt star product \cite{espositostaporwaldmann}]
	\label{theorem:continuity:gutt}
	Let $\mathfrak g$ be a finite-dimensional Lie algebra, $R \geq 1$, and $\hbar \in \mathbb C$.
	Then the Gutt star product $\varstar^{\mathrm G}_{\hbar} \colon \Pol_{\TR}(\mathfrak g^*) \times \Pol_{\TR}(\mathfrak g^*) \to \Pol_{\TR}(\mathfrak g^*)$
	is continuous and therefore extends to a continuous product
	$\varstar^{\mathrm G}_{\hbar} \colon \widehat{\Pol}_{\TR}(\mathfrak g^*) \times \widehat{\Pol}_{\TR}(\mathfrak g^*) \to \widehat{\Pol}_{\TR}(\mathfrak g^*)$.  The maps
	$\mu_{f,g}, \Delta_{f,g} \colon \mathbb C \to \widehat{\Pol}_{\TR}(\mathfrak g^*)$ defined in \eqref{eq:productAndCommutatorMaps} are holomorphic for all $f, g \in \widehat{\Pol}_{\TR}(\mathfrak g^*)$.
\end{theorem}
In Theorems \ref{theorem:continuity:weyl} and \ref{theorem:continuity:gutt} the resulting algebras are not locally multiplicatively convex. In the next sections, we will see that certain homogeneous quadratic Poisson structures can be quantized by locally multiplicatively convex algebras.

\subsection{General continuity results for combinatorial star products} \label{subsec:continuity:combinatorial:general}

In this section we present two general continuity results for combinatorial star products. Our first general result (Proposition \ref{proposition:continuity:formal}) shows that if the combinatorial star product of any homogeneous polynomials of degrees $d_1$ and $d_2$ is a sum of homogeneous polynomials of degrees $\geq d_1 + d_2$, then the product can be extended to all formal power series in the coordinates $x_i$. Examples are quantizations of homogeneous quadratic Poisson structures. Our second general result (Theorem \ref{theorem:continuity:general}) shows that combinatorial star products satisfying some bound on the number of reductions are continuous with respect to the MacGyver topology (Definition \ref{def:norms:general}).

\subsubsection{A coarse topology: The adic topology}
\label{subsubsec:coarse}

Denote the ideal of $\Pol(\mathbb R^d) \simeq \mathbb C [x_1, \dots, x_d]$ that is generated by $x_1, \dots, x_d$ by $\mathfrak m$.
Recall that the $\mathfrak m$\=/adic topology of $\Pol(\mathbb R^d)$ is the one induced by the metric
$d_{\mathfrak m}(f,g) = 2^{-o(f-g)}$ where $o(f-g)$ denotes the largest integer $N \in \mathbb N_0$ such that $f-g \in \mathfrak m^N$, where by convention $o(0) = \infty$ and $d_{\mathfrak m}(f,f) = 0$. 
Note that $\Pol(\mathbb R^d)$ endowed with the $\mathfrak m$\=/adic topology is \emph{not} a topological vector space 
since the scalar multiplication is not continuous.
However, its completion as a metric space is still well-defined and coincides with the space $\mathbb C \llrr {x_1, \dots, x_d}$ of formal power series, which contains $\Pol(\mathbb R^d)$ as a dense subset.

\begin{proposition} \label{proposition:continuity:formal}
	Assume that $\star_{\hbar} \colon \Pol(\mathbb R^d) \times \Pol(\mathbb R^d) \to \Pol(\mathbb R^d)$ only increases the degree
	in the sense that the star product of a homogeneous polynomial of degree $d_1$ and a homogeneous polynomial of degree $d_2$ 
	is the sum of homogeneous polynomials of degrees $\geq d_1+d_2$.
	Then $\star_{\hbar}$ is uniformly continuous with respect to $d_{\mathfrak m}$ and extends to a uniformly continuous product
	$\star_{\hbar} \colon \mathbb C \llrr{x_1, \dots, x_d} \times \mathbb C \llrr{x_1, \dots, x_d} \to \mathbb C \llrr{x_1, \dots, x_d}$
	on formal power series.
\end{proposition}

\begin{proof}
	For $f, g \in \Pol (\mathbb R^d)$ we have $o(f \star_{\hbar} g) \geq o(f) + o(g)$ since $\star_{\hbar}$ only increases the degree, and therefore
	\begin{align*}
	d_{\mathfrak m}(f \star_{\hbar} f', g \star_{\hbar} g') 
	&= 2^{-o(f \star_{\hbar} (f'-g') - (g-f)\star_{\hbar} g')} \\
	&\leq \max\{2^{-o(f \star_{\hbar} (f'-g'))}, 2^{-o((g-f)\star_{\hbar} g')} \} \\
	&\leq \max\{2^{-o(f)} 2^{-o(f'-g')}, 2^{-o(g-f)} 2^{-o(g')} \} \\
	&\leq \max\{d_{\mathfrak m}(f',g'), d_{\mathfrak m}(f,g)\} \\
	&\leq d_{\mathfrak m}^\times((f,f'),(g,g'))
	\end{align*}
	where $d_{\mathfrak m}^\times((f,f'), (g,g')) = d_{\mathfrak m}(f,g) + d_{\mathfrak m}(f',g')$ is the product metric
	on $\mathbb C \llrr{x_1, \dots, x_d} \times \mathbb C \llrr{x_1, \dots, x_d}$. 
	This shows that $\star_\hbar$ is uniformly continuous, and such maps extend uniquely to uniformly continuous maps on the completion.
\end{proof}

Note that the point evaluations $\Pol(\mathbb R^d) \to \mathbb C$, $f \mapsto f(y)$ for $y \in \mathbb R^d \setminus \lbrace 0 \rbrace$ are not continuous with respect to the $\mathfrak m$\=/adic topology, and as a consequence the elements of the completion $\mathbb C\llrr {x_1, \dots, x_d}$ cannot be identified with functions. All topologies on $\Pol(\mathbb R^d)$ considered in the rest of this article will be locally convex, and allow one to identify the elements of the completions with functions. 

\subsubsection{A fine topology: The MacGyver topology}
\label{subsubsec:fine}

A general result for extending $\star_\hbar$ to a space of {\it functions} is given in Theorem~\ref{theorem:continuity:general} below. We first start with a technical lemma, which will be used to show that the dependence of $\star_\hbar$ on $\hbar$ is even better than what is required in our definition of strict deformation quantization (cf.\ Definition \ref{definition:strict}).

\begin{lemma} \label{lemma:uniformEstimates:holomorphicMapsClosed}
	Let $\star_\hbar \colon V \times V \to V$ be a bilinear map on a Fr\'echet space $V$, 
	depending on a parameter $\hbar$ ranging over $\Omega$.
	If $\star_\hbar$ is locally uniformly continuous in $\hbar$,
	in the sense that for every $\hbar_0 \in \Omega$ and every continuous seminorm $\norm \argument_\alpha$ on $V$ there exists an open neighbourhood $U \subset \Omega$ of $\hbar_0$ and a continuous seminorm $\norm \argument_\beta$ on $V$
	with
	\begin{equation*}
		\norm{f \star_\hbar g}_\alpha \leq \norm f_{\beta} \norm g_{\beta}
	\end{equation*} 
	for all $\hbar \in U$ and $f, g \in V$,
	then the set $\set{(f,g) \in V \times V \mid \Omega \ni \hbar \mapsto f \star_\hbar g \in V \textup{ is holomorphic}}$ is closed.
\end{lemma}

\begin{proof}
	Let $\origphi \in V^*$ be a continuous linear functional. 
	Then there is a continuous seminorm $\norm\argument_\alpha$ on $V$ such that $\abs{\origphi(f)} \leq \norm f_\alpha$ holds for all $f \in V$.
	For any $\hbar_0 \in \Omega$ we find an open neighbourhood $U \subset \Omega$ of $\hbar_0$ and a continuous seminorm $\norm\argument_\beta$ on $V$ such that 
	$\norm{ f \star_\hbar g}_\alpha \leq \norm{f}_{\beta} \norm{g}_{\beta}$ holds for all $f, g \in V$ and $\hbar \in U$. 
	Consequently, for any converging sequences $f_n \to f$ and $g_n \to g$ in $V$,
	\begin{equation*}
		\abs{\origphi(f \star_\hbar g) - \origphi(f_n \star_\hbar g_n)} 
		\leq \norm{f \star_\hbar g - f_n \star_\hbar g_n}_\alpha
		\leq \norm{f}_{\beta} \norm{g-g_n}_{\beta} + \norm{g_n}_{\beta} \norm{f-f_n}_{\beta}
	\end{equation*}
	holds for all $\hbar \in U$. So if $\hbar \mapsto f_n \star_\hbar g_n$ is holomorphic for all $n \in \mathbb N$, 
	then this shows that $\hbar \mapsto \origphi(f \star_\hbar g)$ is a uniform limit of holomorphic maps on $U$, 
	hence holomorphic on $U$, and in particular in $\hbar_0$. 
	Since $\hbar_0 \in \Omega$ was arbitrary, $\hbar \mapsto \origphi(f\star_\hbar g)$ is holomorphic on $\Omega$. 
	Thus $\hbar \mapsto f \star_\hbar g$ is holomorphic on $\Omega$, 
	and it follows that the set of $(f,g) \in V \times V$, for which $\hbar \mapsto f \star_\hbar g$ is holomorphic, is closed.
\end{proof}

\begin{remark} \label{remark:uniformEstimates:continuousMapsClosed}
	Using that (locally) uniform limits of continuous functions are again continuous,
	we see that, under the assumptions of Lemma~\ref{lemma:uniformEstimates:holomorphicMapsClosed}, 
	the set of $(f,g) \in V \times V$ for which $\Omega \to V$, $\hbar \mapsto f \star_\hbar g$ is continuous is also closed. 
	This even holds when the domain $\Omega$ is an arbitrary subset of $\mathbb C$.
\end{remark}

\begin{definition}
	\label{def:norms:general} Let the \emph{MacGyver topology} on $\Pol (\mathbb R^d)$ be the locally convex topology given by the family of norms 
	$\lbrace \weakNorm \argument C \mid C \in (0,\infty) \rbrace$ defined by
	\begin{equation} \label{eq:def:sq}
		\weakNorm\argument C \colon \Pol(\mathbb R^d) \to [0, \infty) \komma
		\quad
		\weakNorm[\Big]{\sum_{L \in \mathbb N_0^d} f_L x^L} C \coloneqq \sum_{L \in \mathbb N_0^d} \abs{f_L} C^{{\abs L}^2} \punkt
	\end{equation}
	We write $\weakPol(\mathbb R^d)$ for the resulting locally convex vector space and $\smash{\weak{\widehat{\Pol}{}}(\mathbb R^d)}$ for its completion.
\end{definition}
The MacGyver topology can be viewed a handyperson's first choice for producing strict quantizations with particularly nice dependence on $\hbar$
for a large class of Poisson structures. However, it is finer than the {\TRtop} for any $R \geq 0$,
whence the completion with respect to this topology is strictly smaller than the completion with respect to any of the $\TR$\=/topologies.
A function $f$ is in the completion $\weak{\widehat{\Pol}}(\mathbb R^d)$ if and only if it is analytic and its Taylor coefficients
decay fast enough that the sum on the right of \eqref{eq:def:sq} converges for any $C \in (0,\infty)$.
Any such function extends to a holomorphic function on $\mathbb C^d$.

Recall from Theorem \ref{theorem:graphical} the graphical formula for combinatorial star products given by
$f \star g = \sum_{n \geq 0} \sum_{\Gamma \in \mathfrak G_{n,2}} C_{\Gamma}(f,g)$ 
where $f, g \in \Pol(\mathbb R^d)$ and where $C_{\Gamma} \colon \Pol(\mathbb R^d) \times \Pol(\mathbb R^d) \to \Pol(\mathbb R^d)\llrr t$ 
is a formal series of bidifferential operators, starting with $t^n$ if $\Gamma \in \mathfrak G_{n,2}$. 
Rearranging the sums in powers of $t$, 
we can also write 
\begin{equation}
\label{rearrange}
	f \star g = \sum_{n \geq 0} t^n B_n(f,g)
\end{equation} 
where each $B_n \colon \Pol(\mathbb R^d) \times \Pol(\mathbb R^d) \to \Pol(\mathbb R^d)$ is now a bidifferential operator. It is easy to prove that multiplication and differentiation, and consequently all bidifferential operators, are continuous with respect to the MacGyver topology.
Hence $B_n$ extends to the completion and $B_n(f,g)$ makes sense for all $f,g \in \smash{\weak{\widehat{\Pol}}}(\mathbb R^d)$.

We show that the combinatorial star product is continuous on $\weakPol(\mathbb R^d)$ under a mild finiteness condition:

\begin{theorem} \label{theorem:continuity:general}
	Let $\star$ be a combinatorial star product quantizing a Poisson structure $\eta$ and assume the following:
	\begin{enumerate}[label={\textup{(\alph*)}}]
		\item For any $1 \leq i,j \leq d$ we have
		\begin{equation}
		\label{qijK}
			x_i \star x_j = \sum_{K \in \mathbb N_0^d} q_{i,j,K} x^K
		\end{equation}
		with coefficients $q_{i,j,K} \in \Hol(\Omega)$\footnote{These coefficients are identified with formal power series, as per Convention~\ref{convention}.} and only finitely many coefficients are non-zero.
		\item There is a constant $\alpha$ (independent of $K$ and $L$) such that at most $\alpha (\abs K + \abs L)^2$ many reductions\footnote{For a linear combination of words in $\mathbb C \langle x_1, \dotsc, x_d \rangle$, a {\it reduction} means to perform the right-most possible replacement of $x_j x_i$ by $x_i x_j + \widetilde \phi (x_j x_i)$ if $j > i$ in each word, leaving words of the form $x_1^{J_1} \dotsb x_d^{J_d}$ unchanged (cf.\ Definition \ref{definition:combinatorial}).} are needed to compute $x^K \star x^L$.
		\item There is a constant $\beta$ (independent of $K$ and $L$) such that $x^K \star x^L$ is a sum of monomials of order not greater than $\beta (\abs K + \abs L)$.
	\end{enumerate}
	Then we have the following.
	\begin{enumerate}
	\item 
	For any $\hbar \in \Omega$, the product 
	$\star_\hbar \colon \weakPol(\mathbb R^d) \times \weakPol(\mathbb R^d) 
	\to \weakPol(\mathbb R^d)$, obtained from $\star$ by evaluating $t \mapsto \hbar$, 
	is continuous and extends uniquely to a continuous product
	\[
	\star_\hbar \colon \weak{\widehat{\Pol}{}}(\mathbb R^d) \times \weak{\widehat{\Pol}{}}(\mathbb R^d) 
	\to \weak{\widehat{\Pol}{}}(\mathbb R^d) \punkt
	\]
	\item For any fixed $f, g \in \weak{\widehat{\Pol}{}}(\mathbb R^d)$, the maps $\mu_{f,g}, \Delta_{f,g} \colon \Omega \to \weak{\widehat{\Pol}}(\mathbb R^d)$, defined in \eqref{eq:productAndCommutatorMaps}, are holomorphic.
	\item \label{continuity:general3} If $\complexDisc \subset \Omega$ is a disc around $0$ contained in $\Omega$, then  
	\begin{equation} \label{eq:weak:seriesExpansion}
		\mu_{f,g}(\hbar) = f \star_\hbar g = \sum_{n \geq 0} \hbar^n B_n(f,g)
	\end{equation}
	holds for all $\hbar \in \complexDisc$, where the power series on the right-hand side takes values in the Fr\'echet space $\smash{\weak{\widehat{\Pol}}}(\mathbb R^d)$ and converges absolutely and uniformly on all compact subsets of $\complexDisc$.
	\end{enumerate}
\end{theorem}

\begin{proof}
	Let $U \coloneqq \complexDisc(\hbar_0) \subset \Omega$ be an open disc around some point $\hbar_0 \in \Omega$,
	such that the closed disc $\cc{\complexDisc(\hbar_0)}$ is still contained in $\Omega$.
	Let $Q = N \sup_{i,j,K,\hbar \in U} \lbrace \abs{q_{i,j,K}(\hbar)} \rbrace$, 
	where $N$ is the number of non-zero coefficients $q_{i,j,K}$ in \eqref{qijK}. 
	We can assume that $\alpha, \beta, Q \geq 1$, otherwise just replace the respective constants by $1$ and the same estimates remain true.
	The star product $x^K \star x^L$ can be computed by performing reductions, 
	each of which might increase the number of terms by a factor of $N$ and the coefficients by a factor of $Q/N$.
	Therefore for any $C \geq 1$ and any $\hbar \in U$,
	$$\weakNorm{x^K \star_\hbar x^L}{C} 
	\leq 
	Q^{\alpha (\abs K + \abs L)^2} C^{(\beta (\abs K + \abs L))^2} 
	\leq 
	\big(Q^\alpha C^{\beta^2}\big)^{(\abs K + \abs L)^2}
	\leq 
	\big(Q^\alpha C^{\beta^2}\big)^{2 \abs K^2 + 2 \abs L^2}
	\punkt$$
	For $f = \sum_{K \in \mathbb N_0^d} f_K x^K$ and $g = \sum_{L \in \mathbb N_0^d} g_L x^L$ 
	with only finitely many non-zero coefficients $f_K$ respectively $g_L$,
	we obtain that
	\begin{equation*}
		\weakNorm{f \star_{\hbar} g} C 
		\leq
		\sum_{K,L \in \mathbb N_0^d} \abs{f_K} \abs{g_L} \big(Q^\alpha C^{\beta^2}\big)^{2 \abs K^2 + 2 \abs L^2}
		= 
		\weakNorm{f}{(Q^\alpha C^{\beta^2})^{2}}
		\weakNorm{g}{(Q^\alpha C^{\beta^2})^{2}}
	\end{equation*}
	holds for all $C \geq 1$ and $\hbar \in U$, showing that $\star_\hbar$ is continuous and extends uniquely to the completion by Lemma~\ref{lemma:continuous}.
	
	Moreover, for fixed $f, g \in \Pol(\mathbb R^d)$ all values of the map $\mu_{f,g} \colon \Omega \to \Pol(\mathbb R^d)$ are contained in a finite-dimensional subspace of $\Pol(\mathbb R^d)$ and all its ``components'' are holomorphic on $\Omega$, so $\mu_{f,g}$ is holomorphic for fixed $f,g \in \Pol(\mathbb R^d)$ with absolutely and uniformly convergent power series expansion as in \eqref{eq:weak:seriesExpansion}. Since the estimates above are uniform in $\hbar \in U$, it follows from Lemma~\ref{lemma:uniformEstimates:holomorphicMapsClosed} that $\mu_{f,g}$ is holomorphic for all $f,g \in \smash{\weak{\widehat{\Pol}}}(\mathbb R^d)$.
	As a consequence of Cauchy's integral formula (see e.g.\ \cite[Thm.~3.31]{rudin})
	the power series expansion of $\mu_{f,g}$ around $0$, where $f, g \in \weak{\widehat{\Pol}}(\mathbb R^d)$, is absolutely and uniformly convergent on all compact subsets of any disc $\complexDisc \subset \Omega$ around $0$, and 
	its $n$-th coefficient is
	\begin{equation*}
	\frac 1 {n!} \frac{\partial^n}{\partial \hbar^n} f \star_\hbar g \Big|_{\hbar = 0}
	=
	\frac{1}{2\pi\I} \int_{\partial \complexDisc[r']} \frac{f \star_z g} {z^{n+1}} \mathrm{d} z
	\end{equation*}
	where $0 < r' < r$. Hence the coefficients depend continuously on $(f,g) \in \weak{\widehat{\Pol}}(\mathbb R^d) \times \weak{\widehat{\Pol}}(\mathbb R^d)$ and must therefore coincide with $B_n(f,g)$. 
	Finally, the holomorphy of $\Delta_{f,g}$ follows from that of $\mu_{f,g}$ and the fact that $B_1(f,g) - B_1(g,f) = \I \{f,g\}_\eta$.
\end{proof}

\begin{remark}
	The assumptions of Theorem~\ref{theorem:continuity:general} are fulfilled in all the examples of non-symmetrized combinatorial star products given in \S\ref{sec:starproducts},
	more precisely for the star products from Examples \ref{example:poissonStructure:logcanonicald}, \ref{example:poissonStructure:nonQuadratic}, and \ref{example:poissonStructure:nonQuadratic:exact}.
	Replacing all $x$'s with $w$'s in the definition of the MacGyver topology and Theorem~\ref{theorem:continuity:general},
	the result also holds for the combinatorial star products of Wick type from Example~\ref{example:complexlogcanonical}.
	Note that combinatorial star products with different finiteness conditions, i.e.\ a different bound such as $\alpha(\abs K + \abs L)^3$ for the number of reductions or a different bound for the order of $x^K \star x^L$,
	are typically continuous with respect to an adjusted MacGyver topology, where the exponent of $C$ in \eqref{eq:def:sq} is changed.
\end{remark}

\begin{remark}
\label{remark:smallCompletion}
	The completion $\weak{\widehat{\Pol}}(\mathbb R^d)$ is not as large as one might hope, as it does not contain any non-constant bounded functions $\mathbb R^d \to \mathbb C$.
	Indeed, if $f \in \weak{\widehat{\Pol}}(\mathbb R^d)$ is bounded, its extension to a holomorphic function on $\mathbb C^d$ (which we continue to denote by $f$)
	satisfies $\abs{f(z_1,z_2,\dots,z_d)} \leq C \E^{\abs{z_1}^{1/2}}$ for fixed $z_2, \dots, z_d \in \mathbb C$ and some constant $C > 0$ depending on $z_2, \dots, z_d$, due to the imposed fast decay of the Taylor coefficients.
	By the Phragmén--Lindelöf Theorem (see e.g.~\cite[Ch.~VI, Cor.~4.2]{conway}) this implies that $z_1 \mapsto f(z_1, z_2, \dots, z_d)$ is bounded, hence constant, for all $z_2, \dots, z_d \in \mathbb C$. 
	So $f$ is independent of $z_1$, and we can repeat the argument with the other variables.
\end{remark}

\subsection{Stronger continuity results in examples} \label{subsec:continuity:combinatorial:examples}

In \S\ref{subsec:continuity:combinatorial:general} we saw two topologies which could be used to extend combinatorial star products from the space of polynomial functions to larger spaces. Due to the generality of these results, the completion is in some sense either too large, as the elements can no longer be viewed as smooth functions on $\mathbb R^d$, or too small, as it contains too few interesting functions (cf.\ Remark \ref{remark:smallCompletion}).

In concrete examples, in particular in all of the examples considered thus far, we shall obtain far better continuity estimates for certain choices of the parameters entering the combinatorial star product. These estimates are described in detail in \S\ref{subsubsec:continuity:combinatorial:logcanonical}, \S\ref{subsubsec:continuity:combinatorial:nonquadratic}, and \S\ref{subsec:continuity:symmetrized}, 
after introducing the relevant locally convex topologies in \S\ref{subsubsec:continuity:topologies}. These topologies ensure that the space of polynomial functions is completed to large spaces of functions, so they combine the advantages of the coarse and fine topologies of \S\ref{subsubsec:coarse} and \S\ref{subsubsec:fine}.

\subsubsection{Analytic functions} \label{subsubsec:continuity:topologies}

We begin by introducing certain spaces of analytic functions with Taylor series expansion converging on certain products of intervals $\realDisc$.
For two $d$-tuples $r, r' \in (0, \infty]^d$, we define $r < r'$ and $r \leq r'$ componentwise, 
e.g.\ $r < r'$ holds whenever $r_i < r'_i$ for all $1 \leq i \leq d$.

Let $r = (r_1, \dotsc, r_d) \in (0,\infty]^d$. It is well-known that any holomorphic function $f$ on the polydisc
\begin{equation*}
	\complexDisc \coloneqq \lbrace z \in \mathbb C^d \mid \abs {z_i} < r_i \text{ for all } 1 \leq i \leq d \rbrace
\end{equation*}
has a Taylor series $\sum_{L \in \mathbb N_0^d} f_L z^L$ which converges absolutely and uniformly to $f$ on all compact subsets of $\complexDisc$. Hence the right-hand side of
\begin{equation*}%\label{eq:seminormsAnaInfty}
	\norm{f}_\rho = \sum_{L \in \mathbb N_0^d} \abs{f_L} \rho^{L}
\end{equation*}
converges for any $d$-tuple $\rho \in (0, \infty)^d$ with $\rho < r$ and $\norm f_\rho$ is well-defined for all holomorphic functions $f$ on $\complexDisc$.
We always consider the space $\Hol(\complexDisc)$ of holomorphic functions on $\complexDisc$ 
with the family of norms $\lbrace\norm\argument_\rho \mid \rho \in (0, \infty)^d, \rho < r\rbrace$.

\begin{lemma} \label{lemma:topologiesAgree}
	Let $r \in (0, \infty]^d$. 
	Then the above topology of $\Hol(\complexDisc)$ is precisely the topology of uniform convergence on compact subsets of $\complexDisc$.
\end{lemma}

\begin{proof}
	This result is well-known and a proof is only repeated for the convenience of the reader. We need to show that the families of seminorms 
	$\lbrace \norm \argument_\rho \mid \rho \in (0,\infty)^d, \rho < r \rbrace$ 
	and 
	$\lbrace \sup_{z \in K}\abs{\argument(z)} \mid K \subset \complexDisc , K \text{ compact}\rbrace$ on $\Hol(\complexDisc)$ are equivalent.
	Let $K \subset \complexDisc$ be any compact subset. 
	Then there exists $\rho \in (0, \infty)^d$, $\rho < r$ such that $K \subset \mathbb D_\rho$ and we obtain that
	\begin{equation*}
		\sup_{z \in K} \abs{f(z)} \leq \sup_{z \in K}\sum_{L \in \mathbb N_0^d} \abs{f_L} \abs{z_1}^{L_1} \dots \abs{z_d}^{L_d}
		\leq  \sum_{L \in \mathbb N_0^d} \abs{f_L} \rho^L = \norm f_\rho
	\end{equation*}
	holds for all $f = \sum_{L \in \mathbb N_0^d} f_L z^L \in \Hol(\complexDisc)$.
	Conversely, let $\rho \in (0, \infty)^d$, $\rho < r$ and choose $\rho' \in (0, \infty)^d$, $\rho < \rho' < r$.
	Then we know from Cauchy's integral formula that
	\begin{equation*}
		\abs{f_L} = \frac 1 {L!} \abs{\partial_L f(0)}
		= \frac 1 {(2\pi)^d} \abs[\bigg]{\int_{\partial \mathbb D_{\rho'}} \frac{f(z)}{z^{L+(1,\dots, 1)}} \mathrm{d} z}
		\leq \max_{z \in \partial \mathbb D_{\rho'}} \frac{\abs{f(z)}}{(\rho')^L} 
	\end{equation*}
	where 
	$\partial \mathbb D_{\rho'} 
	= 
	\lbrace z \in \mathbb C^d \mid \abs{z_i} 
	= 
	\rho'_i \text{ for all } 1 \leq i \leq d \rbrace$.
	Consequently,
	\begin{equation*}
		\norm f_\rho 
		= 
		\sum_{L \in \mathbb N_0^d} \abs{f_L} \rho^L 
		\leq 
		\max_{z \in \partial \mathbb D_{\rho'}} \abs{f(z)} \sum_{L \in \mathbb N_0^d} \Big(\frac{\rho}{\rho'}\Big)^L 
		= 
		\max_{z \in \partial \mathbb D_{\rho'}} \abs{f(z)} \frac 1 {1-\frac{\rho_1}{\rho'_1}} \dotsb \frac 1 {1-\frac{\rho_d}{\rho'_d}} \punkt 
	\end{equation*}
	Since $\partial \mathbb D_{\rho'}$ is compact, this shows that the families of seminorms are indeed equivalent.
\end{proof}

It follows from standard results in complex analysis that $\Hol(\complexDisc)$ is a Fréchet space and that the holomorphic polynomials, restricted to $\complexDisc$, form a dense subspace.
Moreover, the inclusion $\iota \colon \realDisc \to \complexDisc$ of 
\begin{equation}
	\realDisc \coloneqq \complexDisc \cap \mathbb R^d 
	= 
	\lbrace 
	x \in \mathbb R^d \mid \abs{x_i} < r_i \text{ for all $1 \leq i \leq d$}
	\rbrace
\end{equation}
into $\complexDisc$ induces an injective map $\iota^* \colon \Hol(\complexDisc) \to \Smooth(\realDisc)$, $f \mapsto f \circ \iota$,
since any holomorphic function is already uniquely determined by its restriction to the reals.
Denote the image of $\iota^*$ by $\Ana(\realDisc)$, and endow it with the locally convex topology that makes $\iota^* \colon \Hol(\complexDisc) \to \Ana(\realDisc)$ a homeomorphism.

\begin{corollary} \label{corollary:topologyOnAnaRd}
	Let $r \in (0,\infty]^d$. Then $\Ana(\realDisc)$ is a Fr\'echet space. 
	Moreover, the polynomials on $\mathbb R^d$ restricted to $\realDisc$ form a dense subspace, 
	and the topology of $\Ana(\realDisc)$ is induced by the family of norms 
	$\lbrace \norm \argument_\rho \mid \rho \in (0, \infty)^d, \rho < r \rbrace$ defined by
	\begin{equation}
		\norm\argument_\rho \colon \Ana(\realDisc) \to [0, \infty) \komma
		\quad
		\norm f_\rho \coloneqq \sum_{L \in \mathbb N_0^d} \abs{f_L} \rho^{L} %\komma
	\end{equation}
	where $\sum_{L \in \mathbb N_0^d} f_L x^L$ is the power series expansion of $f$ around $0$.
\end{corollary}

\begin{definition} \label{def:norm:realanalytic}
	Let $r \in (0,\infty]^d$. 
	We write $\Pol_{r}(\mathbb R^d)$ for the locally convex vector space
	obtained by endowing $\Pol(\mathbb R^d)$ with the family of norms 
	$\lbrace \norm \argument_\rho \mid \rho \in (0, \infty)^d, \rho < r \rbrace$ and set
	$\Pol_\infty(\mathbb R^d) \coloneqq \Pol_{(\infty, \dots, \infty)}(\mathbb R^d)$. As usual, $\widehat{\Pol}_r (\mathbb R^d)$ denotes the completion of $\Pol_r (\mathbb R^d)$.
\end{definition}

\begin{remark}[Comparison of locally convex topologies]
\label{remark:comparison}
The topologies of $\mathcal P_{\TR[0]} (\mathbb R^d)$ and $\mathcal P_\infty (\mathbb R^d)$ coincide, but generally the topology of $\mathcal P_{\TR}(\mathbb R^d)$ is finer than that of $\mathcal P_r (\mathbb R^d)$ and has a smaller completion. More precisely, for $(0, \dots, 0) < r \leq r' \leq (\infty, \dots, \infty)$ and $0 \leq R \leq R' < \infty$, the various completions of $\Pol(\mathbb R^d)$ are related as follows
\begingroup
\small
\[
\weak{\widehat{\Pol}} (\mathbb R^d) \subset \widehat{\Pol}_{\TR[R']} (\mathbb R^d) \subset \widehat{\Pol}_{\TR} (\mathbb R^d) \subset \widehat{\Pol}_{\TR[0]} (\mathbb R^d) = \widehat{\Pol}_{\infty} (\mathbb R^d) \subset \widehat{\Pol}_{r'} (\mathbb R^d) \subset \widehat{\Pol}_r (\mathbb R^d) \subset \mathbb C \llrr{x_1, \dotsc, x_d}.
\]
\endgroup
\end{remark}

\begin{corollary} \label{corollary:requiredContinuityEstimates}
	Let $r \in (0, \infty]^d$.
	A star product $\star \colon \Pol_r(\mathbb R^d) \times \Pol_r(\mathbb R^d) \to \Pol_r(\mathbb R^d)$ is continuous 
	if and only if
	for any $\rho \in (0, \infty)^d$, $\rho < r$ there exists a $\rho' \in (0, \infty)^d$, $\rho' < r$ and a constant 
	$C_{\rho} \in \mathbb R$ such that
	\begin{equation*}
		\norm{f \star g}_{\rho} \leq C_{\rho} \norm f_{\rho'} \norm g_{\rho'} 
	\end{equation*}
	holds for all $f, g \in \Pol(\mathbb R^d)$.
	If this is the case, then $\star$ extends to a unique continuous bilinear map
	$\Ana(\realDisc) \times \Ana(\realDisc) \to \Ana(\realDisc)$, which we typically denote by the same symbol. 
\end{corollary}

\begin{proof}
	Follows directly from Lemma \ref{lemma:continuous} and Corollary~\ref{corollary:topologyOnAnaRd}.
\end{proof}

For some combinatorial star products we can show a stronger statement, 
namely that it suffices to take $\rho' = \rho$ and $C_\rho=1$ in the previous corollary.
In this case, the resulting algebra $(\Ana(\realDisc), \star)$ is locally multiplicatively convex,
meaning that its topology is induced by the seminorms $\norm\argument_\rho$ satisfying $\norm{f \star g}_\rho \leq \norm f_\rho \norm g_\rho$.

\begin{remark}
\label{remark:complextopology}
	For continuity estimates of combinatorial star products of Wick type, we replace the real generators $x_i$ by $w_i$.
	To this end, we let 
	$\wickDisc = \complexDisc \cap \antidiag$ (where $\antidiag$ was defined in \eqref{eq:wickDisc}).
	Note that the inclusion $\kappa \colon \wickDisc \to \mathbb C^d$ still induces an injective map 
	$\kappa^* \colon \Hol(\complexDisc) \to \Smooth(\wickDisc)$.
	Denote the image of $\kappa^*$ by $\Ana(\wickDisc)$, and endow it with the locally convex topology that makes $\kappa^* \colon \Hol(\complexDisc) \to \Ana(\wickDisc)$ a homeomorphism.
	This topology is induced by the family of norms $\lbrace \norm\argument_\rho \mid \rho \in (0, \infty)^d, \rho < r \rbrace$ defined by
	\begin{equation}
		\norm\argument_\rho \colon \Ana(\wickDisc) \to [0, \infty) \komma \quad
		\norm f_\rho \coloneqq \sum_{L \in \mathbb N_0^d} \abs{f_L} \rho^L
	\end{equation}
	where $f = \sum_{L \in \mathbb N_0^d} f_L w^L$ is the power series expansion of $f$ around $0$. 
	Note that $\kappa^* \circ (\iota^*)^{-1} \colon \Ana(\realDisc) \to \Ana(\wickDisc)$ is a homeomorphism,
	mapping $x_i$ to $w_i$. We write $\Pol_{r}(\antidiag)$ for the locally convex vector space obtained by endowing $\Pol(\antidiag)$
	with the family of norms $\lbrace \norm \argument_\rho \mid \rho \in (0, \infty)^d, \rho < r \rbrace$.
\end{remark}

\subsubsection{Log-canonical Poisson structure} \label{subsubsec:continuity:combinatorial:logcanonical}

We saw in Example~\ref{example:poissonStructure:logcanonicald} that there are many different formal star products $\star$ that quantize the log-canonical Poisson structure, depending on a choice of formal power series $q = 1 + \I t + \mathrm{O}(t^2) \in \mathbb C \llrr t$.
If this power series is the formal Taylor expansion of a holomorphic function $q \in \Hol(\Omega)$ around $0$, then $\star$ can be evaluated to a strict associative product $\star_\hbar$ for any $\hbar \in \Omega$, as discussed in \S\ref{subsubsec:convergencelogcanonical}.
Theorem~\ref{theorem:continuity:general} can be used to topologize all these products.
However, if $\abs{q(\hbar)} \leq 1$ then there are much coarser topologies with respect to which $\star_\hbar$ is still continuous, 
allowing us to extend $\star_\hbar$ to much larger function spaces.

\begin{theorem} \label{theorem:continuity:logcanononical:strong}
	Let $q \in \Hol(\Omega)$ be a holomorphic function whose Taylor expansion around $0$ is of the form $1 + \I t + \mathrm{O}(t^2)$,
	and let $\star$ be the associated combinatorial star product quantizing the log-canonical Poisson structure $\eta$ on $\mathbb R^d$ as in Example~\ref{example:poissonStructure:logcanonicald}.
	For any $r \in (0,\infty]^d$ and any $\hbar \in \Omega$ with $\abs{q(\hbar)} \leq 1$ we have:
	\begin{enumerate}
		\item \label{item:theorem:continuity:logcanonical:strong:i}
		The product $\star_{\hbar} \colon \Pol_r(\mathbb R^d) \times \Pol_r(\mathbb R^d) \to \Pol_r(\mathbb R^d)$, obtained from $\star$ by evaluating $t \mapsto \hbar$, is continuous, 
		so that	$\star_{\hbar}$ extends uniquely to a continuous product 
		\[
		\star_{\hbar} \colon \Ana(\realDisc) \times \Ana(\realDisc) \to \Ana(\realDisc) \punkt
		\]
		\item \label{item:theorem:continuity:logcanonical:strong:ii} The resulting algebra $(\Ana(\realDisc), \star_\hbar)$ is locally multiplicatively convex.
		\item \label{item:theorem:continuity:logcanonical:strong:iii}
		For any fixed $f, g \in \Ana(\realDisc)$, the maps $\mu_{f,g}, \Delta_{f,g} \colon \Omega_{\leq 1} \to \Ana(\realDisc)$, defined as in formula \eqref{eq:productAndCommutatorMaps}, are continuous on $\Omega_{\leq 1} = \{\hbar \in \Omega \mid \abs{q(\hbar)} \leq 1\}$
		and holomorphic in the interior $(\Omega_{\leq 1})^\circ = \{\hbar \in \Omega \mid \abs{q(\hbar)} < 1\}$.
	\end{enumerate}
\end{theorem}

\begin{proof}
	Let $f = \sum_{K \in \mathbb N_0^d} f_K x^K$ and $g = \sum_{L \in \mathbb N_0^d} g_L x^L$ with only finitely many non-zero coefficients $f_K$ respectively $g_L$.
	Evaluating the formal deformation parameter $t$ to $\hbar$ (cf.\ \S\ref{subsubsec:convergencelogcanonical}), we obtain that
	\begin{align}
		\norm{f \star_{\hbar} g}_\rho 
		&\leq
		\sum_{K,L \in \mathbb N_0^d} \abs{f_K} \abs{g_L} \norm[\big]{ x^K \star_{\hbar} x^L }_\rho \notag\\
		&=
		\sum_{K,L \in \mathbb N_0^d} \abs{f_K} \abs{g_L} \norm[\Big]{ q(\hbar)^{\sum_{1 \leq i < j \leq d} K_j L_i} x^{K+L} }_\rho \notag\\
		&\leq
		\sum_{K,L \in \mathbb N_0^d} \abs{f_K} \abs{g_L} \rho^{K + L} \notag\\
		&= \norm f_\rho \norm {g}_\rho \label{eq:continuity:logcanonical:estimate}
	\end{align}
	holds for all $\rho \in (0,\infty)^d$, $\rho < r$.
	This shows that the statement of Corollary~\ref{corollary:requiredContinuityEstimates} holds with $\rho' = \rho$ and $C_\rho = 1$,
	implying the continuity of $\star_{\hbar}$ in \refitem{item:theorem:continuity:logcanonical:strong:i},
	but also that the resulting algebra is locally multiplicatively convex, proving \refitem{item:theorem:continuity:logcanonical:strong:ii},
	since all seminorms $\norm\argument_\rho$ are submultiplicative.

	From the explicit formula \eqref{logcanonical} for $\star_{\hbar}$,
	it is clear that $\mu_{f,g}$ and $\Delta_{f,g}$ are continuous on $\Omega_{\leq 1}$ and holomorphic in $(\Omega_{\leq 1})^\circ$
	whenever $f, g \in \Pol(\mathbb R^d)$. Since the estimate \eqref{eq:continuity:logcanonical:estimate} is uniform in $\hbar$,
	the claimed continuity and holomorphy of $\mu_{f,g}$ for $f, g \in \Ana(\realDisc)$ follow from Lemma~\ref{lemma:uniformEstimates:holomorphicMapsClosed} and Remark~\ref{remark:uniformEstimates:continuousMapsClosed}.
	Similarly, the claimed continuity and holomorphy of $\Delta_{f,g}$ for $f, g \in \Ana(\realDisc)$ follows 
	at all points except $0$.
	
	Since $q$ has Taylor expansion $1 + \I t + \mathrm O (t^2)$ around $0$, the map $\hbar \mapsto \frac{1}{\I\hbar} (q(\hbar)-1)$ (extended by $1$ for $\hbar = 0$) is again holomorphic on $\Omega$, in particular bounded on a small neighbourhood $U$ of $0$ by a constant $C \geq 1$, say. For any $\hbar \in U \cap \Omega_{\leq 1}$, we obtain
	\begin{equation*}
	\abs[\bigg]{\frac 1 \hbar (q(\hbar)^{N} - 1)} 
	= \abs[\bigg]{\frac 1 \hbar (q(\hbar) - 1) \sum_{j=0}^{N-1} q(\hbar)^j} 
	\leq C N \punkt
	\end{equation*}
	Consequently, if $\hbar \in U \cap \Omega_{\leq 1} \setminus \lbrace 0 \rbrace$ then
	\begin{multline*}
	\norm{\Delta_{x^K, x^L}(\hbar)}_\rho
	=
	\norm[\bigg]{\frac 1 \hbar (q(\hbar)^{\sum_{1 \leq i < j \leq d} K_j L_i} - q(\hbar)^{\sum_{1 \leq i < j \leq d} L_j K_i}) x^{K+L}}_\rho
	\leq \\ \leq
	\norm[\bigg]{\frac 1 \hbar (q(\hbar)^{\abs{\sum_{1 \leq i < j \leq d} K_j L_i - L_j K_i}}-1) x^{K+L}}_\rho \!\!
	\leq \\ \leq
	C \norm[\bigg]{\sum_{1 \leq i < j \leq d} \!\!(K_j L_i - L_j K_i) x^{K+L}}_\rho \!\!
	=
	C \norm[\big]{ \{x^K, x^L\}_\eta }_\rho \punkt
	\end{multline*}
	Since differentiation and therefore also the Poisson bracket $\{\argument {,} \argument\}_\eta \colon \Pol_r(\mathbb R^d) \times \Pol_r(\mathbb R^d) \to \Pol_r(\mathbb R^d)$ is continuous, there exist $C' > 0$ and $0 < \rho' < r$ such that
	\begin{multline*}
		\norm{\Delta_{f,g}(\hbar)}_\rho 
		\leq \sum_{K,L \in \mathbb N_0^d} \abs{f_K} \abs{g_L} \norm{\Delta_{x^K, x^L}(\hbar)}_\rho
		\leq C \sum_{K,L \in \mathbb N_0^d} \abs{f_K} \abs{g_L} \norm[\big]{\{x^K, x^L\}_\eta}_\rho 
		\leq \\ 
		\leq C' \sum_{K,L \in \mathbb N_0^d} \abs{f_K} \abs{g_L} \norm{x^K}_{\rho'} \norm{x^L}_{\rho'}
		= C' \norm{f}_{\rho'} \norm{g}_{\rho'} \punkt
	\end{multline*}
	Since this estimate is uniform on $U \cap \Omega_{\leq 1}$ and since we also have 
	$\norm{\Delta_{f,g}(0)}_\rho = \norm{\{f,g\}_\eta}_\rho \leq C' \norm f_{\rho'} \norm g_{\rho'}$, 
	the continuity of $\Delta_{f,g}$ in $\hbar = 0$ follows from Remark~\ref{remark:uniformEstimates:continuousMapsClosed}.	
\end{proof}

Note that for $q(\hbar) = \E^{\I \hbar}$, the condition $\abs{q(\hbar)} \leq 1$ is satisfied whenever $\hbar$ lies in the closed upper half-plane,
so in particular for $\hbar \in \mathbb R$. However, when $q(\hbar) = 1+ \I \hbar$, then $0$ is the only real number satisfying $\abs{q(\hbar)} \leq 1$, and we do not obtain a strict quantization in the sense of Definition~\ref{definition:strict} from the previous theorem. This shows that the naive choice of $q$ is not always the best for continuity estimates. If $q(\hbar) = (1-\I\hbar)^{-1}$, then $\mathbb R \setminus \lbrace 0 \rbrace$ is even contained in $(\Omega_{\leq 1})^\circ$, so that the dependence of $f \star_\hbar g$ on $\hbar$ is even analytic on $\mathbb R\setminus \lbrace 0 \rbrace$.

\begin{remark}
Considering $\star$ as in Theorem \ref{theorem:continuity:logcanononical:strong} one can show via more refined arguments that {\it all} derivatives of $\mu_{f,g}$ extend continuously from 
	$(\Omega_{\leq 1})^\circ$ to $\Omega_{\leq 1}$, in particular to $\hbar = 0$. 
	The coefficients $B_n$ in the Taylor expansion $\sum_{n=0}^\infty t^n B_n(f,g)$ 
	of $\mu_{f,g}$ around $\hbar = 0$ coincide with the bidifferential operators \eqref{rearrange} defining $\star$. However, contrary to Theorem~\ref{theorem:continuity:general} (\ref{continuity:general3}),
	the radius of convergence of this Taylor series for general $f, g \in \Ana(\realDisc)$ is $0$.
\end{remark}

\begin{remark}
	The previous proof shows that $\star_\hbar$ is continuous 
	with respect to any norm $\norm\argument_\rho$ with $\rho \in (0, \infty)^d$. 
	Hence it also extends to a continuous product on the Banach space obtained by completing $\Pol(\mathbb R^d)$
	with respect to $\norm\argument_\rho$.
	However, since the topology is determined by a single norm, one sees immediately that the Poisson bracket associated to the log-canonical Poisson structure is not continuous on these spaces.
\end{remark}

\subsubsection{Other polynomial Poisson structures} \label{subsubsec:continuity:combinatorial:nonquadratic}

In this subsection, we show how to extend the star products from Examples~\ref{example:poissonStructure:nonQuadratic} and \ref{example:poissonStructure:nonQuadratic:exact} to spaces of analytic functions.
These results are easy consequences of the following theorem:

\begin{theorem} \label{theorem:continuity:exact}
	Let $N \in \mathbb N_0$, assume that $p,q,r \in \Hol(\Omega)$ are holomorphic functions with $p(0) = q(0) = r(0) = 1$ and $q \neq r^{-N}$,
	and let $\star$ be the associated combinatorial star product introduced in Proposition~\ref{proposition:nonQuadratic}~\refitem{item:proposition:nonQuadraticFormula:i},
	quantizing the Poisson structure \eqref{eq:nonQuadratic:poissonStructure} on $\mathbb R^3$.
	Write
	\begin{equation*}
		\Omega' \coloneqq \Big\lbrace \hbar \in \Omega \mathbin{\Big|} \abs{q(\hbar)} \leq 1,\, \abs{r(\hbar)} = 1, \, \tfrac{p-1}{qr^N-1} \textup{ has at worst a removable singularity at $\hbar$}\Big\rbrace \punkt
	\end{equation*}
	\begin{enumerate}
	\item For any $\hbar \in \Omega'$,
	the product $\star_\hbar \colon \Pol_\infty(\mathbb R^3) \times \Pol_\infty(\mathbb R^3) \to \Pol_\infty(\mathbb R^3)$ obtained by evaluating $\star$ for $t \mapsto \hbar$ as in Proposition~\ref{proposition:nonQuadratic}~\refitem{item:proposition:nonQuadraticFormula:iii} is continuous, 
	so that	$\star_\hbar$ extends uniquely to a continuous product 
	$\star_\hbar \colon \Ana(\mathbb R^3) \times \Ana(\mathbb R^3) \to \Ana(\mathbb R^3)$.
	\item For any fixed $f, g \in \Ana(\mathbb R^3)$, the maps $\mu_{f,g}, \Delta_{f,g} \colon \Omega' \to \Ana(\mathbb R^3)$, defined as in formula \eqref{eq:productAndCommutatorMaps}, are continuous on $\Omega'$ and holomorphic on its interior $(\Omega')^\circ$.
	\end{enumerate}
\end{theorem}

\begin{proof}
	Fix a compact set $K \subset \Omega'$. We claim that there exists a constant $C \geq 1$ such that $\abs{\tilde\lambda_m(w,s)} \leq C$ holds 
	for all $\hbar \in K$, $m \in \mathbb N_0$, $s \in \lbrace 0 , 1 \rbrace$ and words $w$ with letters $0$ and $1$. (Here, $\tilde\lambda_m(w,s)$ is defined as in \eqref{eq:lambda}, but with $p$, $q$, and $r$ evaluated for $t \mapsto \hbar$.) Indeed,
	the fraction on the right-hand side of
	\begin{equation*}
		\abs{\tilde\lambda_m(w,1)} = \abs[\bigg]{(p(\hbar)-1) \sum_{j=0}^{m-\abs w-1} (q(\hbar) r(\hbar)^N)^j} = \abs[\bigg]{\frac{p(\hbar)-1}{q(\hbar)r(\hbar)^N-1}} \abs[\big]{(q(\hbar)r(\hbar)^N)^{m-\abs w}-1}
	\end{equation*}
	is by assumption holomorphic on an open set containing $\Omega'$, and therefore bounded on $K$, 
	and the remaining factor is bounded by $2$ on $\Omega'$. Since $\abs{\tilde \lambda_m(w,0)} \leq 1$ on $\Omega'$, the claim follows.
	Consequently, $\abs{\lambda_m(w)} \leq C^k$ holds for all $k,m \in \mathbb N_0$, $w \in \set{0,1}^k$, and $\hbar \in K$. 
	For any $\rho = (\varrho, \dots, \varrho) \in (0,\infty)^d$ with $\varrho \geq 1$ we estimate
	\begin{align*}
		\norm{x^i y^j z^k \star_\hbar x^\ell y^m z^n}_\rho
		&\leq \sum_{w \in \set{0,1}^k} \abs{r(\hbar)}^{(j-k) \ell + jN\abs w} \abs{\lambda_{m}(w)} 
		\norm[\big]{x^{i+\ell+N\abs w} y^{j+m-\abs w} z^{k+n-\abs w}}_\rho \\
		&\leq \sum_{w \in \set{0,1}^k} C^{k} \varrho^{i+j+k+\ell+m+n + (N-2)\abs w}  \\
		&\leq 2^{k} C^{k} \varrho^{i+j+k+\ell+m+n+N k} \\
		&\leq (2C\varrho^{N+1})^{i+j+k+\ell+m+n}
		\\
		&= \norm{x^i y^j z^k}_{\rho'} \norm{x^\ell y^m z^n}_{\rho'} %\komma
	\end{align*}
	where $\rho' = (2C \varrho^{N+1}, \dots, 2C \varrho^{N+1})$.
	We used that $\varrho^{(N-2) \abs w} \leq \varrho^{N \abs w} \leq \varrho^{N k}$ if $w \in \set{0,1}^k$, 
	and that the set $\set{0,1}^k$ has $2^{k}$ elements.
	Hence we also have
	\begin{align*}
		\norm[\bigg]{\sum_{i,j,k=0}^\infty c_{i,j,k} x^i y^j z^k \star_\hbar \! \sum_{\ell,m,n = 0}^\infty \! d_{\ell,m,n} x^\ell y^m z^n}_{\rho} &\leq \sum_{i,j,k,\ell,m,n=0}^\infty \! \abs{c_{i,j,k}}\abs{d_{\ell,m,n}} \norm{x^i y^j z^k}_{\rho'} \norm{x^\ell y^m z^n}_{\rho'} \\
		&= \norm[\bigg]{\sum_{i,j,k=0}^\infty c_{i,j,k} x^i y^j z^k}_{\rho'} \norm[\bigg]{\sum_{\ell,m,n=0}^\infty \! d_{\ell,m,n} x^\ell y^m z^n}_{\rho'} .%\punkt
	\end{align*}
Note that any point in $\Omega'$ has a compact neighbourhood and that the above estimates are uniform for $\hbar$ in an arbitrary compact subset $K \subset \Omega$, so that the continuity of $\mu_{f,g}$ in $\Omega'$ and of $\Delta_{f,g}$ in $\Omega'\setminus \lbrace 0 \rbrace$,
	as well as the holomorphy of $\mu_{f,g}$ in $(\Omega')^\circ$ and of $\Delta_{f,g}$ in $(\Omega')^\circ \setminus \lbrace 0 \rbrace$,
	follow from Proposition~\ref{proposition:nonQuadratic}~\refitem{item:proposition:nonQuadraticFormula:iv},
	Lemma~\ref{lemma:uniformEstimates:holomorphicMapsClosed}, and Remark~\ref{remark:uniformEstimates:continuousMapsClosed}.
	Continuity respectively holomorphy of $\Delta_{f,g}$ at $\hbar = 0$ (if this point is contained in $\Omega'$ respectively $(\Omega')^\circ$) follow similarly from uniform estimates of $\Delta_{f,g}$ on a small neighbourhood of $0$.
	These estimates can be obtained as above, 
	after noting that the summand for $w = (0,\dotsc,0)$ cancels out in \eqref{eq:productformula:nonQuadratic:exact},
	and that $\abs{\frac 1 \hbar \smash{\tilde\lambda}_m(w,1)} \leq m \smash{\tilde C}$ for some constant $\smash{\tilde C}$ independent of $\hbar \in K$, $m \in \mathbb N_0$, and the word $w$.
\end{proof}
Note that the set $\Omega'$ in the previous theorem depends crucially on $p$, $q$, and $r$. 
It has empty interior (so that the statement about holomorphy is vacuous) unless $r = 1$.

\begin{example}
Choosing $p (\hbar) = q (\hbar) = \E^{\I \hbar}$ and $r (\hbar)=1$ in Theorem \ref{theorem:continuity:exact}, the family $A_\hbar = (\Ana(\mathbb R^3), \star_\hbar)$ defines a strict quantization (in the sense of Definition~\ref{definition:strict}) of the Poisson structure introduced in Example~\ref{example:poissonStructure:nonQuadratic}, where $\hbar$ may assume any value in $\Omega'$ which coincides with the closed upper half-plane. Choosing $q (\hbar) = \frac{1}{1 - \I \hbar}$ instead, we even have that $\Omega'$ contains $\mathbb R \setminus \{ 0 \}$ in its interior. %Good choices would be , in which case $\Omega'$ is the closed upper half-plane, or $p$ an arbitrary entire function with Taylor expansion $1 + \I t + \mathrm O (t^2)$, $q = (1 - \I \hbar)^{-1}$ and $r=1$, in which case $\Omega'$ contains the closed upper half-plane and $(\Omega')^\circ$ contains $\mathbb R\setminus \set 0$.
\end{example}

\begin{example}
Choosing $p(\hbar) = (N + \E^{\I (N+1) \hbar})/(N+1)$ and $q(\hbar) = r(\hbar) = \E^{\I \hbar}$ in Theorem \ref{theorem:continuity:exact}, the family $A_\hbar = (\Ana(\mathbb R^3), \star_\hbar)$ defines a strict quantization of the Poisson structure introduced in Example~\ref{example:poissonStructure:nonQuadratic:exact}.
	In this case, the strict quantization is defined for $\hbar \in \mathbb R$.
	Note that other seemingly natural choices like setting $q(\hbar) = 1+ \I \hbar$ and $p(\hbar) = r(\hbar) = \E^{\I \hbar}$ or setting $q(\hbar) = r(\hbar) = p(\hbar) = 1 + \I \hbar$ lead to $\Omega' \cap \mathbb R = \{ 0 \}$ and hence do not define strict quantizations in the sense of Definition~\ref{definition:strict}.
\end{example}

Another striking special case of Theorem~\ref{theorem:continuity:exact} is the following.

\begin{example}[Quantum Weyl algebra]
	\label{example:quantumWeyl}
	Choosing $p(\hbar) = 1 + \I \hbar$,
	$q(\hbar) = \E^{\I \lambda \hbar}$ for any fixed $\lambda \in (0, \infty)$,
	and $r(\hbar) = 1$ let us apply Theorem \ref{theorem:continuity:exact} for $N=0$ and consider the subspace $\mathbb C [y, z] \subset \mathbb C [x, y, z] = \Pol (\mathbb R^3)$ which is closed under the Poisson bracket $\overline \eta = (\lambda y z + 1) \smash{\frac{\partial}{\partial z}} \wedge \smash{\frac{\partial}{\partial y}}$ obtained by restricting $\eta$ \eqref{eq:nonQuadratic:poissonStructure} to $\mathbb R^2$.
	
%	\comment{revise} In this case, the strict quantization $(\Ana (\mathbb R^3), \star_\hbar)$ is defined for $\hbar \in \Omega' = \overline{\mathbb H} \setminus \lbrace 2 \pi n \lambda^{-1} \mid n \in \mathbb Z \setminus \set 0\rbrace$, where $\overline{\mathbb H}$ denotes the closed upper half-plane, and quantizes the Poisson structure $\eta = (\lambda y z + 1) \smash{\frac{\partial}{\partial z}} \wedge \smash{\frac{\partial}{\partial y}}$ on $\mathbb R^3$. This Poisson structure is independent of the $x$-variable, which is also true for its quantization, in the sense that the star product of $f , g \in \Ana(\mathbb R^3)$ is independent of the $x$-variable if $f$ and $g$ are. Hence we can restrict the strict quantization to $\mathbb R^2$.
	We obtain a combinatorial star product $\star$ on $\mathbb R^2$ associated to the above choices which is determined by
	$\widetilde \phi (z y) = yz (\E^{\I \lambda t} - 1) + \I t$, and hence
	\[
	(\Pol (\mathbb R^2) \llrr{t}, \star) \simeq \mathbb C \langle y,z \rangle \llrr{t} / (z y - \E^{\I \lambda t} y z - \I t)
	\]
	which becomes precisely the so-called \emph{quantum Weyl algebra} when evaluating $t \mapsto \hbar$. 
	Recalling that $\Pol_\infty(\mathbb R^2) = \Pol_{\TR[0]}(\mathbb R^2)$ (Remark \ref{remark:comparison}), Theorem~\ref{theorem:continuity:exact} therefore shows that the resulting strict star product is continuous with respect to the $\TR[0]$\=/topology giving a strict quantization $A_\hbar = (\Ana (\mathbb R^2), \star_\hbar)$.
	
	In particular, this implies continuity with respect to the {\TRtop} for all $R \geq 0$ since $\overline \eta$ is a (non-homogeneous) quadratic Poisson structure. Yet, the standard-ordered Weyl product $\star^{\mathrm{std}}_\hbar$ obtained for $\lambda = 0$ is only continuous with respect to the {\TRtop} for $R \geq \frac12$ (cf.\ Theorem~\ref{theorem:continuity:weyl}). It is quite surprising that in the quantum Weyl algebra, the extra factor of $\E^{\I \lambda \hbar}$ --- which for fixed $\hbar$ can be chosen arbitrarily close to $1$ --- appears to give a much larger strict deformation quantization than the standard Weyl algebra. In fact, this observation can even be generalized to (quantum) Weyl algebras in arbitrary even dimensions.
\end{example}

\begin{remark}
\label{remark:lowerOrderTerms}
A similar phenomenon to Example~\ref{example:quantumWeyl} can also be observed when perturbing the log-canonical Poisson structure by lower orders, i.e.\ by constant and linear terms. Naively, one might expect that the convergence and continuity properties of a formal quantization
of this perturbed Poisson structure should not be better than those of a constant or linear Poisson structure, since it contains constant and linear terms. However, it turns out that any such Poisson structure still admits a strict quantization in the sense of Definition~\ref{definition:strict}
with underlying Fr\'echet space $\Ana(\mathbb R^d) = \widehat{\Pol}_{\TR[0]} (\mathbb R^d)$, contrary to the continuity results for constant or linear Poisson structures (Theorems \ref{theorem:continuity:weyl} and \ref{theorem:continuity:gutt}).

To see this, one may consider the Poisson structure $\eta$ determined by 
\begin{equation*}
	\{x_d, x_1 \}_\eta = x_1 x_d + c \qquad\text{and}\qquad \{x_j, x_i\}_\eta = x_i x_j
\end{equation*}
for $1 \leq i < j \leq d$ with $(i,j) \neq (1,d)$ which can be viewed as a $d$-dimensional generalization of the log-canonical Poisson structure and the Poisson structure of Example \ref{example:quantumWeyl}. This Poisson structure can be quantized by a combinatorial star product $\star$, and convergence and continuity results on the upper half-plane can be obtained in a similar way.

%(which coincides with the Poisson structure of Example~\ref{example:otherPoissonStructureOnRd} \comment{is now really the end of Example 2.37, where $N = 0$} with $q_1 = 1$ and $p_1 = c$). Taking $q(\hbar) = \E^{\I \hbar} \in \Hol(\mathbb C)$ and $p(\hbar) = c \E^{\I \hbar} - c + 1 \in \Hol(\mathbb C)$, one can construct a strict quantization $\star_\hbar \colon \Ana(\mathbb R^d) \times \Ana(\mathbb R^d) \to \Ana(\mathbb R^d)$ of $\eta$ for $\hbar$ in the closed upper half-plane just as in Theorem~\ref{theorem:continuity:exact}, since the extra factor $q^{({\dots})}$ in \eqref{eq:productformula} does not play a role in the continuity estimates.  (Note that with our choices, the set $\Omega'$ occurring in Theorem~\ref{theorem:continuity:exact} is precisely the closed upper half-plane.) 

Next, let $T \colon \Pol(\mathbb R^d) \to \Pol(\mathbb R^d)$ be the algebra homomorphism (with respect to the classical commutative product) determined by $T(x_i) = x_i + c_i$, in other words $T$ is the pull-back with respect to the translation of $\mathbb R^d$ by $(c_1, \dotsc, c_d)$.
It is easy to show that $f \mathbin{\star'} g \coloneqq T(T^{-1} f \star T^{-1} g)$ is again a combinatorial star product,
namely the one defined by $\widetilde\phi' (x_j x_i) = T (\widetilde \phi(x_j x_i))$, and $\star'$ quantizes the Poisson structure $\eta'$ with Poisson bracket determined by
\begin{equation*}
	\{x_d, x_1\}_{\eta'} = x_1 x_d + c_d x_1 + c_1 x_d + c_1 c_d + c \qquad\text{and}\qquad \{x_j, x_i\}_{\eta'} = x_i x_j + c_j x_i + c_i x_j + c_i c_j
\end{equation*}
for $1 \leq i < j \leq d$ with $(i,j) \neq (1,d)$. Since both $T$ and $T^{-1}$ are continuous with respect to the $\TR[0]$\=/topology,
evaluating and completing works in just the same way as for $\star$ and we obtain a strict quantization $\star'_\hbar \colon \Ana(\mathbb R^d) \times \Ana(\mathbb R^d) \to \Ana(\mathbb R^d)$ of $\eta'$, which is defined for all $\hbar$ in the closed upper half-plane.

Finally, checking which restrictions the Jacobi identity imposes on the coefficients of a Poisson structure with Poisson bracket
$\lbrace x_j, x_i \rbrace = x_i x_j + \text{\emph{lower-order terms}}$ for $1 \leq i < j \leq d$ reveals that any such Poisson structure is 
already of the form above for certain $c, c_1, \dots, c_d \in \mathbb R^d$ if $d \neq 3$.
If $d = 3$, then $\{x_d , x_1\}$ may contain an extra summand $c' x_2$ with $c' \in \mathbb R$.
However, a strict quantization in this one remaining case can be constructed by similar methods.
\end{remark}

\subsubsection{Continuity of symmetrized combinatorial star products} \label{subsec:continuity:symmetrized}

We now derive continuity estimates for the symmetrized combinatorial star products introduced in \S\ref{subsubsection:symmetrized}.
Note that the symmetrized Moyal--Weyl product satisfies the same continuity estimates 
	as other non-symmetrized star products for constant Poisson structures 
	because the equivalence transformation between those products is continuous, see \cite[Prop.~5.9]{waldmann1}.
However, the continuity properties of the standard-ordered and the symmetrized combinatorial star products for the log-canonical Poisson structure are different.

As usual, assume that $q \in \Hol(\Omega)$ and $\hbar \in \Omega$. If $q(\hbar)$ is not a root of unity 
(i.e.\ $q(\hbar)^n \neq 1$ for all $n \in \mathbb N$),
then $[k]_{q(\hbar)} = \frac{1 - q(\hbar)^k}{1-q(\hbar)}$ and the $q$-multinomial coefficient can be rewritten as
\begin{equation}
\label{eq:qmultinomial:asFraction}
\begingroup
\binom{\abs K}{K}_{q(\hbar)} = 
\frac{(1-q(\hbar)^{\abs K})(1-q(\hbar)^{\abs K-1}) \dots 
(1-q(\hbar))}{\text{\footnotesize $(1-q(\hbar)^{K_1})(1-q(\hbar)^{K_1-1}) \dots (1-q(\hbar)) \dots (1-q(\hbar)^{K_d})(1-q(\hbar)^{K_d-1}) \dots (1-q(\hbar))$}} \punkt
\endgroup
\end{equation}
If $q(\hbar) \neq 0$, we obtain the identity 
\begin{align} \notag
	\binom{\abs K}{K}_{q(\hbar)} 
	&= \frac{
			(-1)^{\abs K} q(\hbar)^{\abs K + (\abs K-1) + \dots + 1}(1-q(\hbar)^{-\abs K})(1-q(\hbar)^{-\abs K+1}) \dots (1-q(\hbar)^{-1})
		}{
			\prod_{j=1}^d (-1)^{K_j} q(\hbar)^{K_j +\dots + 1}(1-q(\hbar)^{-K_j}) \dots (1-q(\hbar)^{-1}) 
		} \\
	&= q(\hbar)^{\frac 1 2 (\abs{K}^2 - K_1^2 - \dots - K_d^2)} \binom{\abs 
		K}{K}_{q(\hbar)^{-1}}  \punkt \label{eq:qmultinomial:relationOfQandQinverse}
\end{align}

\begin{lemma} \label{lemma:qmultinomials:estimates}
	Let $d \in \mathbb N$ be fixed.
	If $\abs {q(\hbar)} < 1$, then there are constants $c_{q(\hbar)}, C_{q(\hbar)} > 0$ such that
	\begin{equation*}
	c_{q(\hbar)} \leq \abs[\bigg]{\binom{\abs K}{K}_{q(\hbar)}} \leq  C_{q(\hbar)}
	\end{equation*}
	holds for all $K \in \mathbb N_0^d$.
	If $\abs {q(\hbar)} > 1$, then
	\begin{equation*}
	c_{q(\hbar)^{-1}} \abs{q(\hbar)}^{\frac 1 2 (\abs K^2 - K_1^2 - \dots - K_d^2)} 
	\leq 
	\abs[\bigg]{\binom{\abs K}{K}_{q(\hbar)}}
	\leq 
	C_{q(\hbar)^{-1}} \abs {q(\hbar)}^{\frac 1 2 (\abs K^2 - K_1^2 - \dots - K_d^2)}
	\end{equation*}
	holds for all $K \in \mathbb N_0^d$.
\end{lemma}

\begin{proof}
	Since the $q$-Pochhammer symbol $(a;q(\hbar))_\infty = \prod_{k=0}^\infty (1-a q(\hbar)^k)$ is convergent for 
	$\abs {q(\hbar)} < 1$,
	it follows from \eqref{eq:qmultinomial:asFraction} that
	\begin{equation*}
	\frac{(1;\abs {q(\hbar)})_\infty}{(-1;\abs {q(\hbar)})^d_\infty} \leq \abs[\bigg]{\binom{\abs K}{K}_{q(\hbar)}} \leq \frac{(-1;\abs 
		{q(\hbar)})_\infty}{(1;\abs {q(\hbar)})^d_\infty} \punkt
	\end{equation*}
	The case $\abs {q(\hbar)} > 1$ can be reduced to $\abs {q(\hbar)} < 1$ by using 
	\eqref{eq:qmultinomial:relationOfQandQinverse}.
\end{proof}
We can now estimate the coefficients of the symmetrized combinatorial star product for the log-canonical Poisson structure on $\mathbb R^d$.
Introduce the abbreviation 
\begin{equation}
\Lambda_\hbar(K,L) \coloneqq \frac{\binom{\abs K}{K}_{q(\hbar)} \binom{\abs 
		L}{L}_{q(\hbar)}}{\binom{\abs{K+L}}{K+L}_{q(\hbar)}}
q(\hbar)^{\sum_{1 \leq i < j \leq d} K_j L_i}
\end{equation}
so that $x^K \varstar_{\hbar} x^L = \Lambda_\hbar(K,L) \frac{\binom{\abs{K+L}}{K+L}}{\binom{\abs K}{K} \binom{\abs L}{L}} x^{K + L}$.

\begin{lemma}
	With the notation above assume that $\abs{q(\hbar)} \neq 1$.
	Then there is a constant $C_{\Lambda}$ (depending on $\hbar$) such that
	$\abs{\Lambda_\hbar(K,L)} \leq C_{\Lambda}$ holds for all multi-indices $K, L \in \mathbb N_0^d$.
\end{lemma}

\begin{proof}
	If $\abs {q (\hbar)} < 1$ the estimate follows immediately from 
	Lemma~\ref{lemma:qmultinomials:estimates}.
	If $\abs{q(\hbar)} > 1$, then the same lemma implies that 
	\begin{align*}
	\abs{\Lambda_\hbar(K,L)}
	&\leq 
	C_{\Lambda} \abs{q(\hbar)}^{\frac 1 2 (\abs K^2 - \sum_{i=1}^d K_i^2 + \abs L^2 - \sum_{i=1}^d L_i^2 - 
	\abs{K+L}^2 + \sum_{i=1}^d (K+L)_i^2)} \abs {q(\hbar)}^{\sum_{1 \leq i < j \leq d} K_j L_i} 
	\\
	&= C_{\Lambda} \abs {q(\hbar)}^{-\sum_{1 \leq j < i \leq d} K_j L_i} 
	\\
	&\leq C_{\Lambda} %\komma
	\end{align*}
	where $C_{\Lambda} = C_{q(\hbar)^{-1}}^2 / c_{q(\hbar)^{-1}}$ holds for all $K,L \in \mathbb N_0^d$.
\end{proof}

\begin{theorem}
	Let $q \in \Hol(\Omega)$ be a holomorphic function whose Taylor expansion around $0$ is of the form $1 + \I t + \mathrm O(t^2)$,
	and let $\varstar$ be the symmetrized combinatorial star product for the log-canonical Poisson structure on $\mathbb R^d$, 
	introduced in Example~\ref{example:symmetrizedStarProdForLogCanononical}.
	Write $\Omega_{\neq 1} \coloneqq \set{ \hbar \in \Omega \mid \abs{q(\hbar)} \neq 1}$. 
	\begin{enumerate}
	\item For any $\hbar \in \Omega_{\neq 1} \cup \set 0$, 
	the product $\varstar_{\hbar} \colon \Pol_\infty(\mathbb R^d) \times \Pol_\infty(\mathbb R^d) \to \Pol_\infty(\mathbb R^d)$, obtained from $\varstar$ by evaluating $t \mapsto \hbar$, is continuous,
	so that	$\varstar_{\hbar}$ extends uniquely to a continuous product 
	\[
	\varstar_{\hbar} \colon \Ana(\mathbb R^d) \times \Ana(\mathbb R^d) \to \Ana(\mathbb R^d) \punkt
	\]
	\item For any fixed $f, g \in \Ana(\mathbb R^d)$, the maps $\mu_{f,g}, \Delta_{f,g} \colon \Omega_{\neq 1} \to \Ana(\mathbb R^d)$
	are holomorphic.
	\end{enumerate}
\end{theorem}

\begin{proof}
	Using that $\binom{\abs K} K \leq d^{\abs K}$ holds for all $K \in \mathbb N_0^d$, we obtain that
	\begin{equation}
		\norm{x^K \varstar_\hbar x^L}_\rho 
		= \abs[\bigg]{\Lambda_\hbar(K,L) \frac{\binom{\abs{K+L}}{K+L}}{\binom{\abs K}{K} \binom{\abs L}{L}}} \norm{x^{K+L}}_\rho
		\leq C_\Lambda d^{\abs K + \abs L} \norm{x^K}_\rho \norm{x^L}_\rho \leq C_\Lambda \norm{x^K}_{d\rho} \norm{x^L}_{d\rho}
	\end{equation}
	holds for all $\rho \in (0,\infty)^d$.
	The rest of the proof is similar to the proofs given in the previous section.
\end{proof}

\begin{remark}
	Note that the classical limit $\hbar \to 0$ would require continuity of $\mu_{f,g}$ and $\Delta_{f,g}$ at $\hbar = 0$, yet $0 \not\in \Omega_{\neq 1}$.
	In fact, these functions are in general not continuous on $\Omega_{\neq 1} \cup \lbrace 0 \rbrace$, 
	as the poles of $\varstar_\hbar$ accumulate at $0$, so that if $f, g \in \Ana(\mathbb R^d)$ then $f \varstar_\hbar g$ 
	may in general be unbounded on $(\Omega_{\neq 1} \cup \lbrace 0 \rbrace) \cap U$ for any open neighbourhood $U \subset \mathbb C$ of $0$.
	This can be avoided if one restricts to subsets that ``stay far enough away from the poles''. For example, if $q (\hbar) = 1 + \I \hbar$, then one can show that $\mu_{f,g}$ and $\Delta_{f,g}$ are continuous on $\mathbb R \subset \Omega_{\neq 1} \cup \lbrace 0 \rbrace$,
	and we obtain a strict quantization in the sense of Definition~\ref{definition:strict}. 
\end{remark}

\subsection{Positive linear functionals for combinatorial star products of Wick type}
\label{subsec:continuity:wick}

Lastly we consider the combinatorial star product of Wick type of Example~\ref{example:complexlogcanonical}. Recall from Theorem~\ref{theorem:continuity:general} that this star product is continuous on $\weakPol(\antidiag)$ (see Definition~\ref{def:norms:general}). We prove below that the star product is also continuous on $\Pol_r(\antidiag)$ (cf.\ Remark~\ref{remark:complextopology}) which carries a coarser topology and has a larger completion. We then construct an explicit family of positive linear functionals on $\Pol (\antidiag)$,
which extend continuously to $\smash{\weak{\widehat{\Pol}}}(\antidiag)$. This family is point-separating and can therefore be used to faithfully represent the $^*$\=/algebras $\Pol (\antidiag)$ and $\weak{\widehat{\Pol}}(\antidiag)$ with product $\star_\hbar$ and complex conjugation as $^*$\=/involution on a pre-Hilbert space through the GNS-construction. As most of the positive functionals we construct do not extend continuously to $\smash{\widehat\Pol}_r(\antidiag)$, it remains unclear whether this larger algebra admits a faithful representation on some pre-Hilbert space.

The following continuity result for the combinatorial star product of Wick type of Example \ref{example:complexlogcanonical} is proven in the same way as Theorem~\ref{theorem:continuity:logcanononical:strong}, we only have to replace the generators $x_i$ by $w_i$. (Recall that $\wickDisc$ and $\Pol_r(\antidiag)$ were defined in Remark~\ref{remark:complextopology}.)

\begin{theorem}
\label{theorem:strictWick}
	Let $q \in \Hol(\Omega)$ be a holomorphic function whose Taylor expansion around $0$ is of the form $1 - t + \mathrm O (t^2)$,
	and let $\star$ be the associated combinatorial star product of Wick type introduced in Example~\ref{example:complexlogcanonical}.
	For any $r \in (0,\infty]$ and any $\hbar \in \Omega$ such that $\abs{q(\hbar)} \leq 1$ we have the following:
	\begin{enumerate}
		\item 
		The product $\star_{\hbar} \colon \Pol_r(\antidiag) \times \Pol_r(\antidiag) \to \Pol_r(\antidiag)$,
		obtained from $\star$ by evaluating $t \mapsto \hbar$, is continuous, 
		so that	$\star_{\hbar}$ extends uniquely to a continuous product 
		\[
		\star_{\hbar} \colon \Ana(\wickDisc) \times \Ana(\wickDisc) \to \Ana(\wickDisc) \punkt
		\]
		\item 
		The resulting algebra $(\Ana(\wickDisc), \star_\hbar)$ is locally multiplicatively convex.
		\item 
		For any $f, g \in \Ana(\wickDisc)$, the maps $\mu_{f,g}, \Delta_{f,g} \colon \Omega_{\leq 1} \to \Ana(\wickDisc)$, defined as in formula \eqref{eq:productAndCommutatorMaps}, are continuous on $\Omega_{\leq 1} = \{\hbar \in \Omega \mid \abs{q(\hbar)} \leq 1\}$
		and holomorphic in the interior $(\Omega_{\leq 1})^\circ = \{\hbar \in \Omega \mid \abs{q(\hbar)} < 1\}$.
	\end{enumerate}
\end{theorem}

Convenient choices for $q$ are now $q(\hbar) = \E^{-\hbar}$, in which case the quantization is defined if $\RE(\hbar) \geq 0$, or $q(\hbar) = 1-\hbar$, in which case the quantization is defined for $\hbar$ in the closed disc with radius $1$ around $1$. In both cases, we obtain a strict deformation quantization $A_\hbar = (\Ana(\complexDisc), \star_\hbar)$. Note that for both choices all coefficients in the Taylor expansion of $q$ around $\hbar=0$ are real, so that $q(\hbar) \in \mathbb R$ whenever $\hbar \in \mathbb R$. The strict star products $\star_\hbar$ are therefore compatible with the $^*$-involution given by complex conjugation (see Remark \ref{remark:compatibilityWithStar:Wick:strict}) for $\hbar \in [0, \infty)$ and $\hbar \in [0, 2]$, respectively.

In the rest of this section, we study the existence of positive linear functionals aiming at finding ``enough'' positive linear functionals to apply the GNS-construction. To this end we choose $q$ and $\hbar$ such that $q(\hbar) \in (0, \infty)$, as there turn out to be too few positive linear functionals when $q (\hbar) \leq 0$. Note that the product $\star_\hbar$ depends only on $q(\hbar)$, but not the actual value of $\hbar$ itself. Since we are only interested in single values of $q (\hbar)$, we therefore assume without loss of generality for the rest of this section that
\begin{equation}
\label{qhbar}
	q(\hbar) = \E^{-\hbar}
\end{equation}
since this function already attains all values in $(0, \infty)$. Denote the $^*$\=/algebra for this choice of $q$ by $\Pol^\hbar(\antidiag) = (\Pol (\antidiag), \star_\hbar)$. The completions $\Ana^\hbar (\wickDisc)$ and $\widehat\Pol{}^\hbar_{\MG} (\antidiag)$ (with product and involution extended by continuity)
are also $^*$\=/algebras and we have $\Pol^\hbar (\antidiag) \subset \widehat\Pol{}^\hbar_{\MG} (\antidiag) \subset \Ana^\hbar (\wickDisc)$.
\subsubsection{Positivity and continuity of deformed evaluation functionals}

Recall that a linear functional $\omega \colon A \to \mathbb C$, $a \mapsto \dualpairing{\omega}{a}$ 
on a $^*$\=/algebra $A$ is said to be \emph{positive} if $\dualpairing{\omega}{a^* a} \geq 0$ holds for all $a \in A$.

For $z \in \antidiag$, let $\delta_z \colon \Pol(\antidiag) \to \mathbb C$ denote the evaluation functional at $z$,
i.e.\ $\dualpairing{\delta_z}{f} \coloneqq \delta_z(f) = f(z)$ for $f \in \Pol(\antidiag)$.
It is immediate that $\delta_z$ is a positive linear functional for the commutative polynomial algebra $\Pol^0(\antidiag)$ for all $z \in \antidiag$,
but this is no longer true for $\Pol^\hbar (\antidiag)$ with $\hbar > 0$ (and therefore also not true for its completions).

\begin{proposition}
	Let $\hbar > 0$. 
	If $d$ is even, then $\delta_z \colon \Pol^{\hbar}(\antidiag) \to \mathbb C$ with $z \in \antidiag$ is a positive linear functional if and only if $z = 0$. 
	If $d$ is odd, then $\delta_z$ is positive if and only if $z_j = 0$ for all $j \neq \frac 1 2 (d+1)$.
\end{proposition}

\begin{proof}
	Assume that $z \in \antidiag$ and that $z_j \neq 0$ for some $1 \leq j \leq \lfloor \frac 1 2 d \rfloor$. Then $z_{d+1-j} = \cc z_j$ and 
	\begin{align*}
		\dualpairing[\Big]{\delta_z}{\Big(1- \frac{w_j}{z_j}\Big)^* \star_\hbar \Big(1- \frac {w_j}{z_j}\Big)}
		&= 
		\dualpairing[\Big]{\delta_z}{1 - \frac{w_j}{z_j} - \frac{w_{d+1-j}}{z_{d+1-j}} + \E^{-\hbar} \frac{w_j w_{d+1-j}}{z_j z_{d+1-j}} } \\
		&= 1 - 1 - 1 + \E^{-\hbar} \\
		&=
		\E^{-\hbar} - 1
		<
		0 %\komma
	\end{align*}
	showing that $\delta_z$ is not positive. When $d$ is even, the above condition is satisfied for all $z \in \antidiag \setminus \set{0}$, showing that $\delta_z$ is not positive for such points. When $d$ is odd, this shows that $\delta_z$ is not positive whenever $z_j \neq 0$ for some $j \neq \frac 1 2 (d+1)$ and it remains to prove that $\delta_z$ is positive for $z \in \antidiag$ with $z_j = 0$ for all $j \neq \frac 1 2 (d+1)$. For an arbitrary polynomial $f = \sum_{K \in \mathbb N_0^d} f_K w^K$ we have
	\begin{equation}\label{eq:square}
	\begin{aligned}
		f^* \star_\hbar f 
		= \Big(\sum_{K \in \mathbb N_0^d} f_K w^K \Big)^* \star_\hbar \sum_{L \in \mathbb N_0^d} f_L w^L
		&= \sum_{K, L \in \mathbb N_0^d} \cc f{}_K f_L w^{K^\vee} \! \star_\hbar w^L \\
		&= \sum_{K, L \in \mathbb N_0^d} \cc f{}_K f_L \E^{-\hbar \sum_{1 \leq i < j \leq d} K_j^\vee L_i} w^{K^\vee + L}
	\end{aligned}
	\end{equation}
	where $K^\vee = (K_d, \dotsc, K_1)$ and thus
	\begin{multline*}
		\dualpairing{\delta_z}{f^* \star_\hbar f} 
		= \sum_{K, L \in \mathbb N_0^d} \cc f{}_K f_L \E^{-\hbar \sum_{1 \leq i < j \leq d} K_j ^\vee K_i} z^{K^\vee + L}
		= \sum_{k, \ell \in \mathbb N_0} \cc f{}_{k E} f_{\ell E} z_{(d+1)/2}^{k+\ell} = \\
		= \Big(\sum_{k \in \mathbb N_0} f_{k E} z_{(d+1)/2}^{k}\Big)^*
		\Big(\sum_{\ell \in \mathbb N_0} f_{\ell E} z_{(d+1)/2}^{\ell}\Big) 
		\geq 0 \punkt
	\end{multline*}
	Here, $E \in \mathbb N_0^d$ stands for the multi-index that is $1$ in component $(d+1)/2$ and $0$ otherwise.
\end{proof}

On the other hand, for $\hbar < 0$ the situation is different: for $d=2$, the evaluation functionals $\delta_z$ remain positive for all $z \in \antidiag$, and for general $d$ they can be deformed to positive functionals $\delta_z^\hbar \colon \Pol (\antidiag) \to \mathbb C$, defined by extending 
\begin{equation} \label{eq:deltaZHbarForNegativeHbar}
	\dualpairing{\delta_z^{\hbar}}{w^K} = z^K \E^{- \frac 1 2 \hbar \sum_{1 \leq i, j \leq d} m_{ij} K_i K_j}
\end{equation}
linearly, where $m_{ij} \coloneqq \min\lbrace i-1, j-1, d-i, d-j \rbrace$. (Note that $\delta_z^{\hbar} = \delta_z$ if $d=2$.) Indeed, we have the following proposition.
 
\begin{proposition} \label{proposition:functionals:positive}
	Let $\hbar \leq 0$. 
	Then the linear functionals $\delta_z^{\hbar} \colon \Pol^{\hbar}(\antidiag) \to \mathbb C$ are positive for all $z \in \antidiag$.
\end{proposition}

\begin{proof}
	Let $f = \sum_{K \in \mathbb N_0^d} f_K w^K \in \Pol(\antidiag)$. Using \eqref{eq:square}, we compute
	\begin{align*}
		\dualpairing{\delta_z^\hbar}{f^* \star_\hbar f} 
		&= 
		\sum_{K, L \in \mathbb N_0^d} \cc f{}_K f_L z^{K^\vee + L} \E^{-\hbar\sum_{1 \leq i < j \leq d} K_j^\vee L_i - \frac 1 2 \hbar \sum_{1 \leq i, j \leq d} m_{ij} (K^\vee + L)_i (K^\vee + L)_j} \\
		&= 
		\sum_{K, L \in \mathbb N_0^d} \cc v_K M_{KL} v_L
	\end{align*}
	where
	$v_L = f_L z^L \E^{ -\frac 1 2 \hbar \sum_{1 \leq i, j \leq d} m_{ij} L_i L_j}$
	and 
	$M_{KL} = \E^{-\hbar \sum_{1 \leq i < j \leq d} K^\vee_j L_i -\hbar \sum_{1 \leq i, j \leq d} m_{ij} K^\vee_j L_i}$.
	We can rewrite the second sum in the exponent of $M_{KL}$ as
	\begin{align*}
		\sum_{1 \leq i, j \leq d} \! m_{ij} K^\vee_j L_i
		&= \sum_{2 \leq i, j \leq d-1} \! K^\vee_j L_i + \sum_{2 \leq i, j \leq d-2} K^\vee_j L_i + \dotsb \\
		&= \! \sum_{2 \leq j \leq i \leq d-1} \!\! K^\vee_j L_i + \! \sum_{2 \leq i < j \leq d-1} \!\! K^\vee_j L_i
		 + \! \sum_{3 \leq j \leq i \leq d-2} \!\! K^\vee_j L_i + \! \sum_{3 \leq i < j \leq d-2} \!\! K^\vee_j L_i + \dotsb
	\end{align*}
	whence 
	\begin{align*}
		\sum_{1 \leq i < j \leq d} K^\vee_j L_i + \sum_{1 \leq i, j \leq d} m_{ij} K^\vee_j L_i
		&= \sum_{1 \leq i \leq d-1} \sum_{2 \leq j \leq d} K^\vee_j L_i + \sum_{2 \leq i \leq d-2} \sum_{3 \leq j \leq d-1} K^\vee_j L_i + \dotsb \\
		&= \sum_{1 \leq i \leq d-1} \sum_{1 \leq j \leq d-1} K_j L_i + \sum_{2 \leq i \leq d-2} \sum_{2 \leq j \leq d-2} K_j L_i + \dotsb 
	\end{align*}
	and therefore
	\begin{equation*}
			M_{KL} = \E^{-\hbar \sum_{1 \leq j \leq d-1} K_j \sum_{1 \leq i \leq d-1} L_i} \E^{-\hbar \sum_{2 \leq j \leq d-2} K_j \sum_{2 \leq i \leq d-2} L_i} \dotsb \punkt
	\end{equation*}
	Note that for every $n \in \mathbb N_0$ the matrix $(V^n_{ij})_{0 \leq i,j \leq n} \coloneqq \E^{-ij\hbar}$ 
	is a Vandermonde matrix with determinant $\det V^n = \prod_{0 \leq i < j \leq n} (\E^{-j\hbar} - \E^{-i\hbar}) \geq 0$ since $\hbar \leq 0$.
	Consequently (every principal minor of) the infinite matrix $(V_{ij})_{i,j \in \mathbb N_0} \coloneqq \E^{-ij\hbar}$ is positive semidefinite.
	Now $M$ is an entrywise product of infinite matrices, each of which is obtained from $V$ by duplicating certain rows and columns.
	Since (all principal minors of) these factors are positive semidefinite,
	(every principal minor of) the entrywise product $M$ is positive semidefinite.
	But this means that $\dualpairing{\delta_z^\hbar}{\cc f \star_\hbar f} = \sum_{K, L \in \mathbb N_0^d} \cc v_K M_{KL} v_L \geq 0$,
	so $\delta_z^\hbar$ is a positive linear functional.
\end{proof}

The deformed evaluation functionals $\delta_z^\hbar$ for $\hbar < 0$ also give rise to positive linear functionals for $\hbar > 0$ by identifying $\Pol^\hbar (\antidiag)$ with $\Pol^{-\hbar} (\antidiag)$.

	\begin{proposition} \label{proposition:equivalencetransform}
		Let $\hbar > 0$. 
		For every $z \in \antidiag$, the linear functional $\delta_z^\hbar \colon \Pol^\hbar(\antidiag) \to \mathbb C$ determined by 
		\begin{equation}
			\dualpairing{\delta_z^\hbar}{w^K} = z^K \E^{\hbar \sum_{1 \leq i < j \leq d} K_i K_j + \frac 1 2 \hbar \sum_{1 \leq i, j \leq d} m_{ij} K_i K_j}
		\end{equation}
		is positive.
\end{proposition}

\begin{proof}
We claim that the linear map $\Psi_\hbar \colon \Pol^\hbar(\antidiag) \to \Pol^{-\hbar}(\antidiag)$ determined by 
		\begin{equation*}
			\Psi_\hbar(w^K) = \E^{\hbar \sum_{1 \leq i < j \leq d} K_i K_j} w^{K^\vee}
		\end{equation*}
		is a $^*$\=/isomorphism. It is clear that $\Psi_\hbar$ is an isomorphism of vector spaces. Moreover, it intertwines the products $\star_\hbar$ and $\star_{-\hbar}$ since
	\begin{align*}
		\Psi_\hbar(w^K) \star_{-\hbar} \Psi_{\hbar}(w^L) 
		&= \E^{\hbar \sum_{1 \leq i < j \leq d} (K_i K_j + L_i L_j)} w^{K^\vee} \star_{-\hbar} w^{L^\vee} \\
		&= \E^{\hbar \sum_{1 \leq i < j \leq d} (K_i K_j + L_i L_j + K^\vee_j L^\vee_i)} w^{K^\vee + L^\vee} \\
		&= \E^{\hbar \sum_{1 \leq i < j \leq d} ((K + L)_i (K + L)_j - K_j L_i)} w^{K^\vee + L^\vee} \\
		&= \Psi_\hbar(\E^{-\hbar \sum_{1 \leq i < j \leq d} K_j L_i} w^{K + L}) \\
		&= \Psi_\hbar(w^K \star_\hbar w^L)
	\end{align*}
	and it is straightforward to check that $\Psi_\hbar$ also intertwines the $^*$\=/involutions.
	Therefore, the pull-back $\delta_z^\hbar = \delta_{\overline z}^{-\hbar} \circ \Psi_\hbar$ of the positive linear functional $\delta_{\overline z}^{-\hbar} \colon \Pol^{-\hbar}(\antidiag) \to \mathbb C$ (see Proposition~\ref{proposition:functionals:positive}) is positive.
\end{proof}

Proposition \ref{proposition:equivalencetransform} thus yields a family of positive linear functionals on $\Pol^\hbar (\antidiag)$. The next proposition shows continuity results for these positive linear functionals.

\begin{proposition}
\label{proposition:functionals}
	Assume that $\hbar > 0$ and $d \geq 2$. 
	Then $\delta_z^\hbar$ is continuous on $\weakPol(\antidiag)$ for all $z \in \antidiag$ 
	and is continuous on $\Pol_r(\antidiag)$ if and only if $z = 0$.
\end{proposition}

\begin{proof}
	To see that $\delta_z^\hbar$ is continuous on $\weakPol(\antidiag)$, let $c = \max\lbrace 1 , \abs{z_1}, \dots, \abs{z_d} \rbrace$.
	Then we estimate 
	\begin{equation*}
		\abs{\dualpairing{\delta_z^\hbar}{w^K}} 
		= \abs{z^K} \, \E^{\hbar \sum_{1 \leq i < j < d} K_i K_j + \frac 1 2 \hbar \sum_{1 \leq i, j \leq d} m_{ij} K_i K_j}
		\leq (c \, \E^{d\hbar})^{\abs K^2} = \weakNorm{w^K}{c \E^{d \hbar}} \punkt 
	\end{equation*}
	For an arbitrary polynomial $f = \sum_{K \in \mathbb N_0^d} f_{K} w^K$ we find
	$\abs{\dualpairing{\delta_z^\hbar}{f}} \leq \sum_{K \in \mathbb N_0^d} \abs{f_{K}} \weakNorm{w^K}{c \E^{d\hbar}} = \weakNorm{f}{c \E^{d \hbar}}$.
	
	Clearly $\delta_0^\hbar$ is continuous on $\Pol_r(\antidiag)$. 
	Assume that $\delta_z^\hbar$ were continuous on $\Pol_r(\antidiag)$ for some $z \in \antidiag \setminus \set 0$.
	Then there would be $C > 1$ and $0 < \rho < r$ such that
	$\abs{\dualpairing{\delta_z^\hbar}{f}} \leq C \norm f_\rho$ holds for all $f \in \Pol(\antidiag)$.
	Choose $1 \leq i \leq d$ with $z_i \neq 0$ and take $k \in \mathbb N$, $k \geq 2$ large enough 
	such that $\E^{k\hbar/2} \geq C \rho_i \rho_{d+1-i} \abs{z_i}^{-2}$. Then we have
	$\abs{\dualpairing{\delta^q_\hbar}{w_i^k \cc w_i^k}} \geq z_i^k \cc z_i^k \E^{k^2\hbar/2} \geq C^k \rho_i^k \rho_{d+1-i}^{k} = C^k \norm{w_i^k \cc w_i^k}_\rho$,
	contradicting the estimate above and showing that $\delta_z^\hbar$ is not continuous on $\Pol_r(\antidiag)$ if $z \neq 0$.
\end{proof}
\subsubsection{Representation on a pre-Hilbert space}
In the standard formulation of quantum mechanics, the quantum observables are adjointable operators on a (pre)Hilbert space. For a given $^*$\=/algebra $A$, the existence of ``enough'' positive linear functionals gives rise to a faithful representation of $A$ as such a $^*$\=/algebra of adjointable operators via the GNS-construction (see for example \cite[\S 4.4]{schmuedgen}). More precisely, to apply the GNS-construction, one needs a point-separating family of positive linear functionals. We may use the family $\{ \delta_z^\hbar \}_{z \in \antidiag}$ of deformed evaluation functionals which by Proposition \ref{proposition:functionals} extend to $\smash{\weak{\widehat{\Pol}}}(\antidiag)$ and are still positive. It remains to show that the family $\{ \delta_z^\hbar \}_{z \in \antidiag}$ is point-separating on $\smash{\weak{\widehat{\Pol}}}(\antidiag)$.

\begin{lemma} \label{lemma:pointSeparating}
Let $\hbar \in \mathbb R$.
		Then the family of linear functionals $\set{\delta_z^\hbar}_{z \in \antidiag}$ separates the points of $\weak{\widehat{\Pol}}(\antidiag)$, 
		i.e.\ for any $f \in \weak{\widehat{\Pol}}(\antidiag)$, 
		there exists $z \in \antidiag$ such that $\dualpairing{\delta_z^\hbar}{f} \neq 0$.
\end{lemma}

\begin{proof}
	Let $\hbar \geq 0$, $f = \sum_{K \in \mathbb N_0^d} f_{K} w^K \in \weak{\widehat{\Pol}}(\antidiag)$ and assume that
	$\dualpairing{\delta_z^\hbar}{f} = 0$ holds for all $z \in \antidiag$. 
	Since $\delta_z^\hbar$ is continuous on $\weak{\widehat{\Pol}}(\antidiag)$, 
	\begin{equation*}
	0 = \dualpairing{\delta_z^\hbar}{f} = \sum_{K \in \mathbb N_0^d} f_K z^K \E^{\hbar\sum_{1 \leq i < j < d} K_i K_j + \frac 1 2 \hbar \sum_{1 \leq i, j \leq d} m_{ij} K_i K_j}
	\end{equation*}
	where the power series on the right-hand side converges for all $z \in \antidiag$. 
	Hence all coefficients $f_K$ with $K \in \mathbb N_0^d$ must vanish.
	The argument for $\hbar < 0$ is analogous in which case one may use the formula \eqref{eq:deltaZHbarForNegativeHbar} to compute $\dualpairing{\delta_z^\hbar}{f}$.
\end{proof}

The existence of a point-separating family of positive linear functionals on $\weak{\widehat{\Pol}}(\antidiag)$ allows us to deduce the following result.

\begin{theorem}
\label{theorem:GNS}
Let $\weak{\widehat{\Pol}{}^\hbar}(\antidiag)$ be the strict deformation quantization given in Theorem \ref{theorem:strictWick}
 with $q(\hbar) = \E^{-\hbar}$. For any fixed $\hbar \in \mathbb R$ there exists a pre-Hilbert space $\mathcal D$ and an injective $^*$\=/homomorphism
\[
\weak{\widehat{\Pol}{}^\hbar}(\antidiag) \to \mathcal L^* (\mathcal D)
\]
where $\mathcal L^* (\mathcal D)$ denotes the $^*$\=/algebra of adjointable operators on $\mathcal D$.
\end{theorem}

\begin{proof}
The linear functionals $\delta_z^\hbar$ separate the points of $\weak{\widehat{\Pol}{}^\hbar}(\antidiag)$ by Lemma~\ref{lemma:pointSeparating},
and are all positive by Propositions~\ref{proposition:functionals:positive} and \ref{proposition:equivalencetransform}.
A faithful representation of $\weak{\widehat{\Pol}{}^\hbar}(\antidiag)$
on a pre-Hilbert space $\mathcal D$ can now be obtained via the GNS-construction \cite[\S 4.4]{schmuedgen}.
\end{proof}

\begin{remark}
\label{remark:GNS}
It would be interesting to determine whether the larger algebras $\Ana^\hbar (\wickDisc)$ for $\hbar > 0$ can also be represented faithfully on a pre-Hilbert space. Since $\delta_z^\hbar$ extends continuously to $\Ana(\wickDisc)$ only if $\hbar = 0$, it currently remains unclear whether $\Ana^\hbar(\wickDisc)$ admits a point-separating family of positive linear functionals.
\end{remark}

\subsection{Further directions}

We close with some remarks on possible directions for extending our results.

\subsubsection*{Other Poisson structures}

In all of our examples, as well as in Theorem \ref{theorem:continuity:general}, we looked at quantizations of polynomial Poisson structures for which only finitely many reductions are needed for all fixed polynomials $f, g$. This condition is satisfied for many well-known classes of polynomial Poisson structures, including constant and linear Poisson structures, as well as several higher-order Poisson structures, notably the log-canonical Poisson structures. However, one can also construct examples which do not satisfy this finiteness condition.

We expect that at least for general quadratic Poisson structures on $\mathbb R^d$, the coefficients which appear in the combinatorial star product $x^K \star x^L$ of two general monomials are rational functions in the parameters $q_{ji}^{k\ell}$, where $x_j \star x_i = \sum_{1 \leq k \leq \ell \leq d} q^{k\ell}_{ji} x_k x_\ell$ for $1 \leq i < j \leq d$. In other words, the formula for $x^K \star x^L$ may pick up extra poles, similarly to what we saw for the symmetrized combinatorial star product (see \S\ref{subsec:continuity:symmetrized}). Nevertheless, we expect similar convergence and continuity results to hold for all $\hbar$ in some dense subset of the domain of definition of the $q_{ji}^{k\ell}$'s. We thus expect our results to generalize to much larger classes of combinatorial star products. Recall that any polynomial Poisson structure can be quantized via combinatorial star products.

On the other hand, we also saw in \S\ref{subsec:continuity:combinatorial:examples} that the strongest results could be achieved for concrete examples. In order to construct strict quantizations which are defined on large function spaces it thus seems prudent to work directly with particular (classes of) Poisson structures of interest, as it may be difficult to obtain strong results for general polynomial Poisson structures.

For example, one might look to construct strict deformation quantizations of certain classes of log symplectic structures. The Poisson structures associated to the local normal form of log symplectic structures contain only linear and constant terms \cite[Thm.~37]{guilleminmirandapires} and can thus readily be quantized by the combinatorial star product (see Proposition \ref{proposition:moyalWeylAndGutt}). For degree reasons, the star product converges on polynomials and we expect the continuity estimates in this case to be similar to those of purely constant or purely linear Poisson structures (see \S\ref{subsec:continuity:weylandgutt}). However, there are also log symplectic structures on $\mathbb R^d$ which degenerate along a hypersurface of higher degree, for example along a smooth real elliptic (i.e.\ cubic) curve in $\mathbb R^2$ \cite[Ex.~1.13]{gualtierili}. In case the hypersurface is defined by polynomial equations, we expect the combinatorial approach described in \S\ref{subsection:combinatorial} to yield formal quantizations. However, higher-degree Poisson structures may not satisfy the finiteness condition on the number of reductions (cf.\ the hypotheses of Theorem \ref{theorem:continuity:general}), so some further work will be needed to obtain convergence and continuity results for these examples.

\subsubsection*{Representation-theoretic aspects of strict deformation quantizations}

In Theorem \ref{theorem:GNS} we obtained a faithful representation of a strict deformation quantization of Wick type on a pre-Hilbert space through the GNS-construction which allows one to use also operator-theoretic tools in its study. It would be interesting to see whether this result can also be generalized to larger $^*$\=/algebras (cf.\ Remark \ref{remark:GNS}), other combinatorial star products of Wick type, or whether a similar result can be shown for the symmetrized combinatorial star products introduced in \S\ref{subsubsection:symmetrized}.

\subsubsection*{Comparison with C$^*$-algebraic approach to strict quantization}

The construction of strict deformation quantizations can be viewed as a further step towards a concrete comparison between deformation quantization and other approaches to strict quantizations such as C$^*$-algebraic quantizations obtained for example via quantum groups, as initiated by M.A.~Rieffel \cite[\S12]{rieffel}. In general such a comparison is difficult, because algebraic constructions of quantizations usually contain only unbounded functions, whereas C$^*$-algebraic quantizations contain only bounded elements. However, the strict deformation quantization $(\Ana (\mathbb C), \star_\hbar)$ constructed in \S\ref{subsec:continuity:wick} contains on the one hand the ``algebraic'' quantum plane $\mathbb C \langle x, y \rangle / (y x - q (\hbar) x y)$ as a subalgebra, and on the other hand bounded functions such as $\mathrm e^{-x^2}$, say, which should also be contained in C$^*$-algebraic strict quantizations. This could facilitate and simplify concrete comparisons between these two approaches to strict quantizations.

\paragraph*{Acknowledgements}

It is a pleasure to thank Jonas Schnitzer for helpful discussions and for making us acquainted in the first place, and Martin Bordemann, Matthias Schötz, Stefan Waldmann and the anonymous referees for various valuable comments.

The first named author also gratefully acknowledges the support by the Deutsche Forschungsgemeinschaft (DFG, German Research Foundation) through the DFG Research Training Group GK1821 ``Cohomological Methods in Geometry'' at the University of Freiburg.\\[\topsep]
{\footnotesize This version of the article has been accepted for publication, but is not the Version of Record and does not reflect post-acceptance improvements, or any corrections. The Version of Record is freely available online at: \url{https://doi.org/10.1007/s00220-022-04541-4}}

%\paragraph*{Data availability}
%Data sharing not applicable to this article as no datasets were generated or analyzed as part of this research.

\paragraph*{Competing interests} The authors declare that they have no competing interests.

\end{document}